\documentclass[12pt]{article}

\usepackage{amsmath}
\usepackage{amssymb}
\usepackage{amsthm}
\usepackage[mathscr]{eucal}
\usepackage[all]{xy}

\newtheorem{theorem}{Theorem}[section]
\newtheorem{lemma}[theorem]{Lemma}
\newtheorem{proposition}[theorem]{Proposition}
\newtheorem{corollary}[theorem]{Corollary}

\theoremstyle{definition}
\newtheorem{definition}[theorem]{Definition}
\newtheorem{example}[theorem]{Example}

\theoremstyle{remark}
\newtheorem{remark}[theorem]{Remark}

\newtheorem{claim}{Claim}
\newtheorem*{claim*}{Claim}

\title{Reflections and torsion theories for selfinjective algebras}

\author{Hiroki Abe}
\date{}

\begin{document}

\maketitle


\begin{abstract}
We introduce the notion of reflections for selfinjective algebras from the point of view of torsion theories induced by two-term tilting complexes. As an application, we determine the transformations of Brauer trees associated with reflections. In particular, we provide a way to transform every Brauer tree into a Brauer line. 
\end{abstract}

\section{Introduction}


Reflection functors introduced in \cite{BGP} are induced by transformations of the quiver making a certain source vertex changed into a sink vertex. Let $\Lambda$ be a finite dimensional algebra over a field $K$. In \cite{APR}, it was shown that reflection functors are of the form $\mathrm{Hom}_{\Lambda}(T,-)$ with $T$ a certain type of tilting modules. Let $P_{1}, \cdots, P_{n}$ be a complete set of nonisomorphic indecomposable projective modules in $\mathrm{mod}\text{-}\Lambda$, the category of finitely generated right $\Lambda$-modules. Set $I=\{1, \cdots , n\}$. Assume that there exists a simple projective module $S \in \mathrm{mod}\text{-}\Lambda$ which is not injective. Take $t \in I$ with $P_{t} \cong S$ and set 
\[
T = T_{1} \oplus \tau^{-1}S \quad \text{with} \quad T_{1} = \bigoplus_{i \in I \setminus \{ t \} } P_{i}, 
\]
where $\tau$ denotes the Auslander-Reiten translation. Then $T$ is a tilting module, called an APR-tilting module, and $\mathrm{Hom}_{\Lambda}(T,-)$ is a reflection functor. \\
\indent
In \cite{BB}, APR-tilting modules were generalized as follows. Assume that there exists a simple module $S \in \mathrm{mod}\text{-}\Lambda$ such that $\mathrm{Ext}^{1}_{\Lambda}(S,S)=0$ and $\mathrm{Hom}_{\Lambda}(D\Lambda, S)=0$, where $D=\mathrm{Hom}_{K}(-,K)$. Let $P_{t}$ be the projective cover of $S$ and let $T$ be the same as above. Then $T$ is a tilting module, called a BB-tilting module. It is well-known that $T$ induces a torsion theory for $\mathrm{mod}\text{-}\Lambda$ whose torsion-free class is the full subcategory consisting of direct sums of copies of $S$. \\
\indent
We are interested in a minimal projective presentation of $T$, which is a two-term tilting complex. Take a minimal injective presentation $0 \to S \to E^{0} \overset{f}{\to} E^{1}$ and define a complex $E^{\bullet}$ as the mapping cone of $f: E^{0} \to E^{1}$. Then $\mathrm{Hom}^{\bullet}_{\Lambda}(D\Lambda, E^{\bullet})$ is a minimal projective presentation of $\tau^{-1}S$ and hence
\[
T^{\bullet} = T_{1} \oplus \mathrm{Hom}^{\bullet}_{\Lambda}(D\Lambda, E^{\bullet})
\]
is a minimal projective presentation of $T$. In this note, we demonstrate that this type of tilting complexes play an important role in the theory of derived equivalences for selfinjective algebras. \\
\indent
Consider the case where $\Lambda$ is selfinjective and $S \in \mathrm{mod}\text{-}\Lambda$ is a simple module with $\mathrm{Ext}^{1}_{\Lambda}(S,S)=0$ and $\mathrm{Hom}_{\Lambda}(D\Lambda, S) \cong S$. Let $E^{\bullet}$ and $T^{\bullet}$ be the same as above. We will show that $T^{\bullet}$ is a tilting complex and $T^{\bullet} \cong T_{1} \oplus E^{\bullet}$. Also, we will show that $T^{\bullet}$ induces a torsion theory for $\mathrm{mod}\text{-}\Lambda$ whose torsion-free class is the full subcategory consisting of direct sums of copies of $S$. In this note, derived equivalences for selfinjective algebras induced by this type of tilting complexes are called reflections. Finally, as an application, we will determine the transformations of Brauer trees associated with reflections. \\
\indent
We refer to \cite{HR} for the definition and basic properties of tilting modules, to \cite{Har} and \cite{Ve} for basic results in the theory of derived categories and to \cite{Ri1} for definitions and basic properties of tilting complexes and derived equivalences.

\section{Tilting complex for selfinjective algebras}

Throughout this note, $K$ is a commutative artinian local ring and $\Lambda$ is an Artin $K$-algebra, i.e., $\Lambda$ is a ring endowed with a ring homomorphism $K \to \Lambda$ whose image is contained in the center of $\Lambda$ and $\Lambda$ is finitely generated as a $K$-module. We always assume that $\Lambda$ is connected, basic and not simple. We denote by $\mathrm{mod}\text{-}\Lambda$ the category of finitely generated right $\Lambda$-modules and by $\mathcal{P}_{\Lambda}$ the full subcategory of $\mathrm{mod}\text{-}\Lambda$ consisting of projective modules. For a module $M \in \mathrm{mod}\text{-}\Lambda$, we denote by $P(M)$ (resp., $E(M)$) the projective cover (resp., injective envelope) of $M$. We denote by $\mathscr{K}(\mathrm{mod}\text{-}\Lambda)$ the homotopy category of cochain complexes over $\mathrm{mod}\text{-}\Lambda$ and by $\mathscr{K}^{\mathrm{b}}(\mathcal{P}_{\Lambda})$ the full triangulated subcategory of $\mathscr{K}(\mathrm{mod}\text{-}\Lambda)$ consisting of bounded complexes over $\mathcal{P}_{\Lambda}$. We consider modules as complexes concentrated in degree zero. For an object $A$ in an additive category $\mathfrak{A}$, we denote by $\mathrm{add}(A)$ the full subcategory of $\mathfrak{A}$ consisting of direct summands of finite direct sums of copies of $A$. \\
\indent
Throughout the rest of this note, we assume that $\Lambda$ is selfinjective. Let $S \in \mathrm{mod}\text{-}\Lambda$ be a simple module with $\mathrm{Ext}^{1}_{\Lambda}(S,S)=0$ and $E(S) \cong P(S)$. Note that $E(S) \cong P(S)$ if and only if $\mathrm{Hom}_{\Lambda}(D\Lambda, S) \cong S$, where $D$ denotes the Matlis dual over $K$. Take a minimal injective presentation $0 \to S \to E^{0} \overset{f}{\to} E^{1}$ and define a complex $E^{\bullet} \in \mathscr{K}^{\mathrm{b}}(\mathcal{P}_{\Lambda})$ as the mapping cone of $f: E^{0} \to E^{1}$. Note that $E^{1}$ is the $0$th term of $E^{\bullet}$ and $E^{0}$ is the $(-1)$th term of $E^{\bullet}$. Let $P_{1}, \cdots, P_{n}$ be a complete set of nonisomorphic indecomposable modules in $\mathcal{P}_{\Lambda}$ and set $I= \{ 1, \cdots, n \}$. We assume that $n >1$. Take $t \in I$ with $P_{t} \cong P(S)$ and set
\[
T^{\bullet} = T_{1} \oplus E^{\bullet} \quad \text{with} \quad T_{1} = \bigoplus_{i \in I \setminus \{ t \} } P_{i}.
\]
The following holds. 

\begin{theorem}\label{tilting_complex}
The complex $T^{\bullet} \in \mathscr{K}^{\mathrm{b}}(\mathcal{P}_{\Lambda})$ is a tilting complex for $\Lambda$ and $\mathrm{End}_{\mathscr{K}(\mathrm{mod}\text{-}\Lambda)}(T^{\bullet})$ is a selfinjective Artin $K$-algebra whose Nakayama permutation coincides with that of $\Lambda$. 
\end{theorem}
\begin{proof}
We use the notation $\mathrm{Hom}^{\bullet}(-,-)$ to denote the single complex associated with the double hom complex. For a subcategory $\mathfrak{S}$ of a triangulated category $\mathfrak{T}$, we denote by $\langle\mathfrak{S}\rangle$ the full triangulated subcategory of $\mathfrak{T}$ generated by $\mathfrak{S}$. 
\begin{claim}\label{serre_duality} 
For any $j \in \mathbb{Z}$ we have a functorial isomorphism 
\[
\mathrm{Hom}_{\mathscr{K}(\mathrm{mod}\text{-}\Lambda)}(X^{\bullet}, E^{\bullet}[j]) \cong D\mathrm{Hom}_{\mathscr{K}(\mathrm{mod}\text{-}\Lambda)}(E^{\bullet}, X^{\bullet}[-j])
\]
for $X^{\bullet} \in \mathscr{K}(\mathrm{mod}\text{-}\Lambda)$. 
\end{claim}
\begin{proof}
Set the Nakayama functor $\nu = - \otimes_{\Lambda}D\Lambda$. For $Q \in \mathcal{P}_{\Lambda}$ and $X \in \mathrm{mod}\text{-}\Lambda$, we have a bifunctorial isomorphism 
\[
X \otimes_{\Lambda}\mathrm{Hom}_{\Lambda}(Q,\Lambda) \overset{\sim}{\to} \mathrm{Hom}_{\Lambda}(Q,X), x \otimes h \mapsto (y \mapsto xh(y))
\]
and hence by adjointness we have bifunctorial isomorphisms 
\begin{align*}
D\mathrm{Hom}_{\Lambda}(Q,X) &\cong D(X \otimes_{\Lambda}\mathrm{Hom}_{\Lambda}(Q,\Lambda)) \\
 &\cong \mathrm{Hom}_{\Lambda}(X, \nu Q). 
\end{align*} 
Thus, for $Q^{\bullet} \in \mathscr{K}^{\mathrm{b}}(\mathcal{P}_{\Lambda})$ and $X^{\bullet} \in \mathscr{K}(\mathrm{mod}\text{-}\Lambda)$, we have a bifunctorial isomorphism in $\mathscr{K}(\mathrm{mod}\text{-}K)$ 
\[
D\mathrm{Hom}^{\bullet}_{\Lambda}(Q^{\bullet},X^{\bullet}) \cong \mathrm{Hom}^{\bullet}_{\Lambda}(X^{\bullet}, \nu Q^{\bullet})
\]
and hence we have bifunctorial isomorphisms  
\begin{align*}
D\mathrm{Hom}_{\mathscr{K}(\mathrm{mod}\text{-}\Lambda)}(Q^{\bullet},X^{\bullet}) &\cong D\mathrm{H}^{0}(\mathrm{Hom}^{\bullet}_{\Lambda}(Q^{\bullet},X^{\bullet}))  \\
 &\cong \mathrm{H}^{0}(D\mathrm{Hom}^{\bullet}_{\Lambda}(Q^{\bullet},X^{\bullet})) \\
 &\cong \mathrm{H}^{0}(\mathrm{Hom}^{\bullet}_{\Lambda}(X^{\bullet}, \nu Q^{\bullet})) \\
 &\cong \mathrm{Hom}_{\mathscr{K}(\mathrm{mod}\text{-}\Lambda)}(X^{\bullet}, \nu Q^{\bullet}). 
\end{align*} 
Now, since $E(S) \cong P(S)$ implies $E^{\bullet} \cong \nu E^{\bullet}$, for any $j \in \mathbb{Z}$ we have a functorial isomorphism 
\[
\mathrm{Hom}_{\mathscr{K}(\mathrm{mod}\text{-}\Lambda)}(X^{\bullet}, E^{\bullet}[j]) \cong D\mathrm{Hom}_{\mathscr{K}(\mathrm{mod}\text{-}\Lambda)}(E^{\bullet}, X^{\bullet}[-j])
\]
for $X^{\bullet} \in \mathscr{K}(\mathrm{mod}\text{-}\Lambda)$. 
\end{proof}

\begin{claim}\label{E_E}
$\mathrm{Hom}_{\mathscr{K}(\mathrm{mod}\text{-}\Lambda)}(E^{\bullet}, E^{\bullet}[j]) = 0$ for $j \ne 0$. 
\end{claim}
\begin{proof}
It is obvious that $\mathrm{Hom}_{\mathscr{K}(\mathrm{mod}\text{-}\Lambda)}(E^{\bullet}, E^{\bullet}[j]) = 0$ unless $-1 \le j \le 1$. Since $\mathrm{Hom}_{\Lambda}(S,E^{1}) \cong \mathrm{Ext}^{1}_{\Lambda}(S,S)=0$, the homomorphism $\mathrm{Hom}_{\Lambda}(f,E^{1})$ is surjective and hence every $g \in \mathrm{Hom}_{\Lambda}(E^{0},E^{1})$ factors through $f$, which implies $\mathrm{Hom}_{\mathscr{K}(\mathrm{mod}\text{-}\Lambda)}(E^{\bullet}, E^{\bullet}[1]) = 0$. It then follows by Claim \ref{serre_duality} that 
\begin{align*}
\mathrm{Hom}_{\mathscr{K}(\mathrm{mod}\text{-}\Lambda)}(E^{\bullet}, E^{\bullet}[-1]) &\cong D\mathrm{Hom}_{\mathscr{K}(\mathrm{mod}\text{-}\Lambda)}(E^{\bullet}, E^{\bullet}[1]) \\
 &= 0.
\end{align*} 
\end{proof}

\begin{claim}\label{E_P}
If $ i \ne t$, then $\mathrm{Hom}_{\mathscr{K}(\mathrm{mod}\text{-}\Lambda)}(E^{\bullet}, P_{i}[j]) = 0$ for $j \ne 0$. 
\end{claim}
\begin{proof}
It is obvious that $\mathrm{Hom}_{\mathscr{K}(\mathrm{mod}\text{-}\Lambda)}(E^{\bullet}, P_{i}[j]) = 0$ unless $0 \le j \le 1$. Also, since $\mathrm{Hom}_{\Lambda}(S, P_{i}) = 0$, and since $P_{i}$ is injective, $\mathrm{Hom}_{\Lambda}(f, P_{i})$ is surjective and $\mathrm{Hom}_{\mathscr{K}(\mathrm{mod}\text{-}\Lambda)}(E^{\bullet}, P_{i}[1]) = 0$.
\end{proof}

\begin{claim}\label{P_E}
If $ i \ne t$, then $\mathrm{Hom}_{\mathscr{K}(\mathrm{mod}\text{-}\Lambda)}(P_{i}, E^{\bullet}[j]) = 0$ for $j \ne 0$. 
\end{claim}
\begin{proof}
This follows by Claims \ref{serre_duality} and \ref{E_P}. 
\end{proof}

\begin{claim}\label{generate}
$\langle\mathrm{add}(T^{\bullet})\rangle = \mathscr{K}^{\mathrm{b}}(\mathcal{P}_{\Lambda})$. 
\end{claim}
\begin{proof}
Since $P_{i} \in \mathrm{add}(T^{\bullet})$ for $i \ne t$, we have only to show that $P_{t} \in \langle\mathrm{add}(T^{\bullet})\rangle$. Note that $P_{t} \cong P(S) \cong E^{0}$. Also, since $\mathrm{Ext}^{1}_{\Lambda}(S,S)=0$, $E^{0}$ is not a direct summand of $E^{1}$. Thus $E^{1} \in \mathrm{add}(T_{1}) \subset \mathrm{add}(T^{\bullet})$ and hence, since $E^{\bullet} \in \mathrm{add}(T^{\bullet})$, we have $E^{0} \in \langle\mathrm{add}(T^{\bullet})\rangle$. 
\end{proof}
This finishes the proof of Theorem \ref{tilting_complex}. 
\end{proof}

Next, we recall some definitions and results on stable torsion theories induced by two-term tilting complexes (see \cite{Ab} and \cite{HKM} for details).  Let $A$ be an Artin algebra. 

\begin{definition}\label{def_stable_torsion_theory}
A pair $(\mathcal{T},\mathcal{F})$ of full subcategories $\mathcal{T}$, $\mathcal{F}$ in $\mathrm{mod}\text{-}A$ is said to be {\it a torsion theory} for $\mathrm{mod}\text{-}A$ if the following conditions are satisfied.
\begin{enumerate}
\item[(1)] $\mathcal{T} \cap \mathcal{F} = \{ 0 \}$.  
\item[(2)] $\mathcal{T}$ is closed under factor modules. 
\item[(3)] $\mathcal{F}$ is closed under submodules. 
\item[(4)] For any $X \in \mathrm{mod}\text{-}A$, there exists an exact sequence $0 \to X^{\prime} \to X \to X^{\prime \prime} \to 0$ with $X^{\prime} \in \mathcal{T}$ and $X^{\prime \prime} \in \mathcal{F}$.
\end{enumerate}
In particular, $\mathcal{T}$ (resp., $\mathcal{F}$) is said to be a torsion (resp., torsion-free) class. If $\mathcal{T}$ is stable under the Nakayama functor $\nu = - \otimes_{A} DA$, then $(\mathcal{T},\mathcal{F})$ is said to be {\it stable}. If each indecomposable module in $\mathrm{mod}\text{-}A$ lies either in  $\mathcal{T}$ or in  $\mathcal{F}$, then $(\mathcal{T},\mathcal{F})$ is said to be {\it splitting}. 
\end{definition}

Let $P^{\bullet} \in \mathscr{K}^{\mathrm{b}}(\mathcal{P}_{A})$ be a two-term complex: 
\[
P^{\bullet} : \cdots \to 0 \to P^{-1} \overset{\alpha}{\to} P^{0} \to 0 \to \cdots .          
\]
We set the following subcategories in $\mathrm{mod}\text{-}A$: 
\[
\mathcal{T}(P^{\bullet}) = \mathrm{Ker}\ \mathrm{Hom}_{\mathscr{K}(\mathrm{mod}\text{-}A)}(P^{\bullet}[-1], - ) \cap \mathrm{mod}\text{-}A, 
\]
\[
\mathcal{F}(P^{\bullet}) = \mathrm{Ker}\ \mathrm{Hom}_{\mathscr{K}(\mathrm{mod}\text{-}A)}(P^{\bullet}, - ) \cap \mathrm{mod}\text{-}A. 
\]

\begin{proposition}[{\cite[Proposition 5.5]{HKM}}]\label{tilting_stable_torsion_theory}
The following are equivalent. 
\begin{enumerate}
\item[(1)] $P^{\bullet}$ is a tilting complex. 
\item[(2)] $(\mathcal{T}(P^{\bullet}),\mathcal{F}(P^{\bullet}))$ is a stable torsion theory for $\mathrm{mod}\text{-}A$. 
\end{enumerate}
\end{proposition}

\begin{proposition}\label{torsion_free_class}
Let $T^{\bullet}$ be the tilting complex constructed in Theorem \ref{tilting_complex}. Then we have a stable torsion theory $(\mathcal{T}(T^{\bullet}),\mathcal{F}(T^{\bullet}))$ for $\mathrm{mod}\text{-}\Lambda$ satisfying the following conditions. 
\begin{enumerate}
\item[(1)] $\mathcal{F}(T^{\bullet})=\mathrm{add}(S)$. 
\item[(2)] $P(S)$ lies neither in $\mathcal{T}(T^{\bullet})$ nor in $\mathcal{F}(T^{\bullet})$. In particular, $(\mathcal{T}(T^{\bullet}),\mathcal{F}(T^{\bullet}))$ is not splitting.  
\end{enumerate}
\end{proposition}
\begin{proof}
Since $T^{\bullet}$ is a two-term tilting complex of $\Lambda$, the pair $(\mathcal{T}(T^{\bullet}),\mathcal{F}(T^{\bullet}))$ is a stable torsion theory for $\mathrm{mod}\text{-}\Lambda$ by Proposition \ref{tilting_stable_torsion_theory}. \\
\indent
(1) Since $\mathrm{H}^{0}(T^{\bullet})=T_{1} \oplus {\tau}^{-1}S$ with $T_{1} = \bigoplus_{i \in I \setminus \{ t \} } P_{i}$ and $\mathrm{H}^{-1}({\nu}T^{\bullet})=S$, we have 
\begin{align*}
\mathcal{F}(T^{\bullet}) &= \mathrm{Ker}\ \mathrm{Hom}_{\mathscr{K}(\mathrm{mod}\text{-}\Lambda)}(T^{\bullet}, - ) \cap \mathrm{mod}\text{-}\Lambda \\
 &= \mathrm{Ker}\ \mathrm{Hom}_{\Lambda}(T_{1} \oplus {\tau}^{-1}S, - ) \\
 &= \mathrm{cog}(S), 
\end{align*}
where $\mathrm{cog}(S)$ is the full subcategory of $\mathrm{mod}\text{-}A$ whose objects are cogenerated by $S$. Since $S$ is simple, we have $\mathrm{cog}(S) = \mathrm{add}(S)$. The assertion follows. \\
\indent
(2) Assume first that $P(S) \in \mathcal{T}$. Since $\bigoplus_{i \in I \setminus \{ t \}} P_{i} = T_{1} \in \mathcal{T}$ and $P_{t}=P(S)$, $\mathcal{T}$ contains all indecomposable projective module in $\mathrm{mod}\text{-}\Lambda$ and hence $\mathcal{T} = \mathrm{mod}\text{-}\Lambda$. It is a contradiction by (1). Next, assume that $P(S) \in \mathcal{F}$. Then $P(S) \cong S$ by (1). Therefore $S$ is a simple projective module. It is a contradiction because $\Lambda$ is selfinjective and connected. 
\end{proof}

For a quiver $Q$, we denote by $\alpha_{n-1}\alpha_{n-2}\cdots\alpha_{1}$ the path in $Q$ 
\[
i_{1} \overset{\alpha_{1}}{\to} \cdots \overset{\alpha_{n-2}}{\to} i_{n-1} \overset{\alpha_{n-1}}{\to} i_{n}.
\]
If $\Lambda$ is the path algebra defined by $Q$ over a field and $P_{i}$ is the indecomposable module in $\mathcal{P}_{\Lambda}$ corresponding to the vertex $i$, then we have a sequence of homomorphisms in $\mathcal{P}_{\Lambda}$ 
\[
P_{i_{1}} \overset{\alpha_{1}}{\to} \cdots \overset{\alpha_{n-2}}{\to} P_{i_{n-1}} \overset{\alpha_{n-1}}{\to} P_{i_{n}}
\]
such that $\alpha_{j} \in \mathrm{rad}(\Lambda) \setminus \mathrm{rad}^{2}(\Lambda)$ for $j=1, \cdots , n-1$.

\begin{example}\label{example_torsion_theory}
Let $\Lambda$ be the path algebra defined by the quiver 
\[
\xymatrix{
1 \ar [r]<2pt>^{\alpha_{1}} & 2 \ar [r]<2pt>^{\alpha_{2}} \ar [l]<2pt>^{\beta_{1}}  & 3 \ar [l]<2pt>^{\beta_{2}} 
}
\]
with relations $\alpha_{2}\alpha_{1}=\beta_{1}\beta_{2}=0$ and $\alpha_{1}\beta_{1} = \beta_{2}\alpha_{2}$. Then $\Lambda$ is selfinjective with $\nu=\mathrm{id}_{\mathrm{mod}\text{-}\Lambda}$ and the Auslander--Reiten quiver of $\Lambda$ is given by the following: 
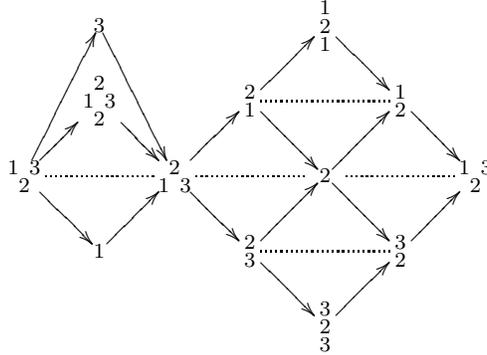
\[
\begin{xy}
(-40,20) *{\begin{smallmatrix} \ 3 \ \end{smallmatrix}}="S3",
(-40,10) *{\begin{smallmatrix} 2 \\ 1 \ 3 \\ 2 \end{smallmatrix}}="P2",
(-50,0) *{\begin{smallmatrix} 1 \ 3 \\ 2 \end{smallmatrix}}="JP2",
(-40,-10) *{\begin{smallmatrix} \ 1\  \end{smallmatrix}}="S1",
(-30,0) *{\begin{smallmatrix} 2 \\ 1 \ 3 \end{smallmatrix}}="P2S",
(-20,10) *{\begin{smallmatrix} 2 \\ 1 \end{smallmatrix}}="JP1",
(-20,-10) *{\begin{smallmatrix} 2 \\ 3 \end{smallmatrix}}="JP3",
(-10,0) *{\begin{smallmatrix} \ 2 \ \end{smallmatrix}}="S2",
(-10,20) *{\begin{smallmatrix} 1 \\ 2 \\ 1 \end{smallmatrix}}="P1",
(-10,-20) *{\begin{smallmatrix} 3 \\ 2 \\ 3 \end{smallmatrix}}="P3",
(0,10) *{\begin{smallmatrix} 1 \\ 2 \end{smallmatrix}}="P1S",
(0,-10) *{\begin{smallmatrix} 3 \\ 2 \end{smallmatrix}}="P3S",
(10,0) *{\begin{smallmatrix} 1\ 3 \\ 2 \end{smallmatrix}}="REJP2",
\ar "JP2" ; "S3"  
\ar "JP2" ; "P2" 
\ar "JP2" ; "S1"
\ar "S3" ; "P2S" 
\ar "P2" ; "P2S" 
\ar "S1" ; "P2S" 
\ar@{.} "JP2" ; "P2S"
\ar "P2S" ; "JP1"  
\ar "P2S" ; "JP3"  
\ar "JP1" ; "S2" 
\ar "JP3" ; "S2" 
\ar "JP1" ; "P1" 
\ar "P1" ; "P1S" 
\ar "JP3" ; "P3"
\ar "P3" ; "P3S"  
\ar@{.} "P2S" ; "S2"
\ar@{.} "JP1" ; "P1S"
\ar@{.} "JP3" ; "P3S"
\ar "S2" ; "P1S"  
\ar "S2" ; "P3S"  
\ar "P1S" ; "REJP2" 
\ar "P3S" ; "REJP2"
\ar@{.} "S2" ; "REJP2"  
\end{xy}
\]
where each indecomposable module is represented by its composition factors and $\tau$-orbits are denoted by $\xymatrix{\bullet \ar@{.} [r] & \bullet}$. Since the exact sequence 
\[
0 \to \begin{smallmatrix} 1 \end{smallmatrix} \to \begin{smallmatrix} 1 \\ 2 \\ 1 \end{smallmatrix} \to \begin{smallmatrix} 2 \\ 1 \ 3 \\ 2 \end{smallmatrix}
\] 
give a minimal injective presentation of the simple module corresponding to the vertex $1$, we have a tilting complex $T^{\bullet}=T_{1} \oplus E^{\bullet}$ with 
\[
T_{1}= 0 \to \begin{smallmatrix} 2 \\ 1 \ 3 \\ 2 \end{smallmatrix} \oplus \begin{smallmatrix} 3 \\ 2 \\ 3 \end{smallmatrix} \quad \text{and} \quad E^{\bullet}= \begin{smallmatrix} 1 \\ 2 \\ 1  \end{smallmatrix} \to \begin{smallmatrix} 2 \\ 1 \ 3 \\ 2 \end{smallmatrix} 
\]
by Theorem \ref{tilting_complex}. 
It is not difficult to see that $T^{\bullet}$ induces the stable torsion theory
\[
\mathcal{T}(T^{\bullet})=\{\begin{smallmatrix} 2 \\ 1 \ 3 \\ 2 \end{smallmatrix}, \begin{smallmatrix} 2 \\ 1 \ 3 \end{smallmatrix}, \begin{smallmatrix} 2 \\ 3 \end{smallmatrix}, \begin{smallmatrix} 2 \end{smallmatrix}, \begin{smallmatrix} 3 \\ 2 \\ 3 \end{smallmatrix}, \begin{smallmatrix} 3 \\ 2 \end{smallmatrix} \} \quad \text{and} \quad \mathcal{F}(T^{\bullet}) = \{\begin{smallmatrix} 1 \end{smallmatrix} \},
\]
for $\mathrm{mod}\text{-}\Lambda$. Note that the indecomposable projective module $\begin{smallmatrix} 1 \\ 2 \\ 1  \end{smallmatrix}$ which is a projective cover of the simple module $\begin{smallmatrix} 1 \end{smallmatrix}$ lies neither in $\mathcal{T}(T^{\bullet})$ nor in $\mathcal{F}(T^{\bullet})$. 
\end{example}

A torsion theory whose torsion-free class is the full subcategory consisting of direct sums of copes of a simple module characterizes a reflection functor which induced by an APR-tilting module or a BB-tilting module. From the point of view of torsion theories, we introduce the notion of reflections for selfinjective algebras .  

\begin{definition}\label{reflection}
Let $T^{\bullet}$ be the tilting complex constructed in Theorem \ref{tilting_complex}. The derived equivalence induced by $T^{\bullet}$ is said to be the reflection for $\Lambda$ at $t$. Sometimes, we also say that $\mathrm{End}_{\mathscr{K}(\mathrm{mod}\text{-}\Lambda)}(T^{\bullet})$ is the reflection of $\Lambda$  at $t$. 
\end{definition}

\begin{example}\label{example_tilting}
Let $\Lambda$ be the path algebra defined by the quiver
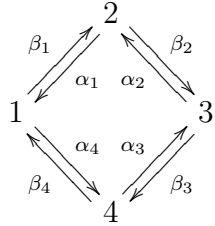
\[
\xymatrix{
 & 2 \ar [ld]<2pt>^{\alpha_{1}} \ar [rd]<2pt>^{\beta_{2}} & \\
1 \ar [ru]<2pt>^{\beta_{1}} \ar [rd]<2pt>^{\alpha_{4}} & & 3 \ar [ld]<2pt>^{\beta_{3}} \ar[lu]<2pt> ^{\alpha_{2}}  \\
 & 4 \ar [lu]<2pt>^{\beta_{4}} \ar [ru]<2pt>^{\alpha_{3}}&
}
\]
with relations
\[
\beta_{2}\beta_{1}=\alpha_{3}\alpha_{4},\ \beta_{4}\beta_{3}=\alpha_{1}\alpha_{2},\ \beta_{1}\alpha_{1}=\alpha_{2}\beta_{2},\ \beta_{3}\alpha_{3}=\alpha_{4}\beta_{4},
 \]
\[
\alpha_{4}\alpha_{1}=\alpha_{2}\alpha_{3}=\beta_{1}\beta_{4}=\beta_{3}\beta_{2}=\alpha_{1}\beta_{1}=\beta_{2}\alpha_{2}=\alpha_{3}\beta_{3}=\beta_{4}\alpha_{4}=0.
\]
Then $\Lambda$ is selfinjective. For the simple module $S$ corresponding to the vertex $2$, we have $E(S) \cong P(S)$ and $\mathrm{Ext}^{1}_{\Lambda}(S,S)=0$. Thus by Theorem \ref{tilting_complex} we have a tilting complex $T^{\bullet}$ and another selfinjective algebra $\Gamma=\mathrm{End}_{\mathscr{K}(\mathrm{mod}\text{-}\Lambda)}(T^{\bullet})$ which is the reflection of $\Lambda$ at $2$. It is not difficult to see that $\Gamma$ is the path algebra defined by the quiver 
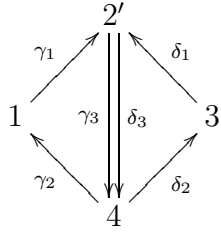
\[
\xymatrix{
 & 2^{\prime} \ar [dd]<-2pt>_{\gamma_{3}} \ar [dd]<2pt>^{\delta_{3}} & \\
1 \ar [ru]^{\gamma_{1}} & & 3 \ar [lu]_{\delta_{1}}  \\
 & 4 \ar [lu]^{\gamma_{2}} \ar [ru]_{\delta_{2}} &
}
\]
with relations
\[
\gamma_{1}\gamma_{2}\gamma_{3}=\delta_{1}\delta_{2}\delta_{3},\ \delta_{3}\gamma_{1}\gamma_{2}=\gamma_{3}\delta_{1}\delta_{2},
\]
\[
\gamma_{3}\gamma_{1}=\delta_{3}\delta_{1}=\gamma_{2}\delta_{3}=\delta_{2}\gamma_{3}=\gamma_{3}\delta_{1}\delta_{2}\delta_{3}=\delta_{3}\gamma_{1}\gamma_{2}\gamma_{3}=0,
\]
 where $2^{\prime}$ is the vertex corresponding to $E^{\bullet}$. 
\end{example}

\section{Brauer tree algebras}

Throughout this section, we assume that $K$ is an algebraically closed field. Recall that a Brauer tree $(B, v, m)$ consists of a finite tree $B$, called the underlying tree,  together with a distinguished vertex $v$, called the exceptional vertex and a positive integer $m$, called the multiplicity. In case $m=1$, $(B,v,m)$ is identified with the underlying tree $B$ and is called a Brauer tree without exceptional vertex. The pair of the number of edges of $B$ and the multiplicity $m$ is said to be the numerical invariants of $(B, v, m)$. Each Brauer tree determines a symmetric $K$-algebra $\Lambda$ up to Morita equivalence (see \cite{Alp} for details), called a Brauer tree algebra, which is given as the path algebra defined by some quiver with relations $(\Lambda_{0}, \Lambda_{1}, \rho)$, where $\Lambda_{0}$ is the set of vertices, $\Lambda_{1}$ is the set of arrows between vertices and $\rho$ is the set of relations (see \cite{GR} for details). We have the following.  

\begin{proposition}\label{der_eq_Brauer_tree}
Let $\Lambda$ be a Brauer tree algebra. Then every ring $\Gamma$ derived equivalent to $\Lambda$ is a Brauer tree algebra having the same numerical invariants as $\Lambda$. 
\end{proposition}
\begin{proof}
Note that $\Lambda$ and $\Gamma$ are stably equivalent to each other (see \cite[Theorem 4.2]{Ri2}, \cite{Ri3}). Then we know from \cite{GR} that $\Gamma$ is given by some Brauer tree which has the same numerical invariants as $\Lambda$. 
\end{proof}

In this section, we will apply Theorem \ref{tilting_complex} to Brauer tree algebras and determine the transformations of Brauer tree algebras induced by reflections.

\begin{remark}\label{nakayama_functor}
Let $\Lambda$ be a  Brauer tree algebra. Then for any simple module $S \in \mathrm{mod}\text{-}\Lambda$ we have $E(S) \cong P(S)$. 
\end{remark}

Throughout the rest of this section, we deal only with Brauer trees without exceptional vertex. Let $\Lambda$ be a Brauer tree algebra, $(\Lambda_{0}, \Lambda_{1}, \rho)$ the quiver with relations of $\Lambda$ and $t \in \Lambda_{0}$. We denote by the diagram 
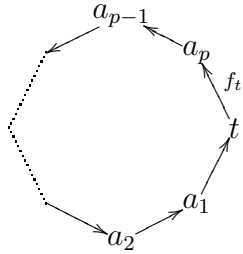
\[
\begin{xy}
(15,-15) *{t}="t",
(10,-25) *{a_{1}}="a(1)",
(0,-30) *{a_{2}}="a(2)",
(-10,-25) *{}="a(3)",
(-15,-15) *{}="a(i)",
(-10,-5) *{}="a(p-2)", 
(0,0) *{a_{p-1}}="a(p-1)",
(10,-5) *{a_{p}}="a(p)", 
\ar "a(1)" ; "t"  
\ar "a(2)" ; "a(1)" 
\ar "a(3)" ; "a(2)" 
\ar @{.} "a(i)" ; "a(3)"  
\ar @{.} "a(p-2)" ; "a(i)" 
\ar "a(p-1)" ; "a(p-2)"  
\ar "a(p)" ; "a(p-1)" 
\ar_{f_{t}} "t" ; "a(p)" 
\end{xy}
\]
with $p \ge 1$ the situation that $t$ belongs to at most one cycle, and by the diagram 
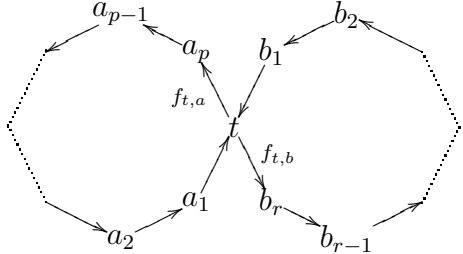
\[
\begin{xy}
(15,-15) *{t}="t",
(10,-25) *{a_{1}}="a(1)",
(0,-30) *{a_{2}}="a(2)",
(-10,-25) *{}="a(3)",
(-15,-15) *{}="a(i)",
(-10,-5) *{}="a(p-2)", 
(0,0) *{a_{p-1}}="a(p-1)",
(10,-5) *{a_{p}}="a(p)", 
(20,-5) *{b_{1}}="b(1)",
(30,0) *{b_{2}}="b(2)",
(40,-5) *{}="b(3)",
(45,-15) *{}="b(i)",
(40,-25) *{}="b(r-2)",
(30,-30) *{b_{r-1}}="b(r-1)",
(20,-25) *{b_{r}}="b(r)",

\ar "a(1)" ; "t"  
\ar "a(2)" ; "a(1)" 
\ar "a(3)" ; "a(2)" 
\ar @{.} "a(i)" ; "a(3)"  
\ar @{.} "a(p-2)" ; "a(i)" 
\ar "a(p-1)" ; "a(p-2)"  
\ar "a(p)" ; "a(p-1)" 
\ar^{f_{t,a}} "t" ; "a(p)" 
\ar "b(1)" ; "t" 
\ar "b(2)" ; "b(1)" 
\ar "b(3)" ; "b(2)" 
\ar @{.} "b(i)" ; "b(3)" 
\ar @{.} "b(r-2)" ; "b(i)" 
\ar "b(r-1)" ; "b(r-2)" 
\ar "b(r)" ; "b(r-1)" 
\ar^{f_{t,b}} "t" ; "b(r)" 

\end{xy}
\]
with $p, r \ge 1$ the situation that $t$ belongs to two cycles. We denote by $S_{t}$ the simple module corresponding to $t$ and by $P_{t}$ the projective cover of $S_{t}$. 

\begin{lemma}\label{self_ext}
The following hold. 
\begin{enumerate}
\item[(1)] If $t$ belongs to at most one cycle, we have a minimal injective presentation $0 \to S_{t} \to P_{t} \overset{f}{\to} P_{a_{p}}$ with $f=f_{t}$. 
\item[(2)] If $t$ belongs to two cycle, we have a minimal injective presentation 
\[
0 \to S_{t} \to P_{t} \overset{f}{\to} P_{a_{p}} \oplus P_{b_{r}} \ \text{with} \ f = \left(\begin{array}{c}
f_{t,a} \\
f_{t,b}
\end{array}\right).
\]
\item[(3)] For any $t \in \Lambda_{0}$, we have $\mathrm{Ext}^{1}_{\Lambda}(S_{t}, S_{t})=0$. 
\end{enumerate}
\end{lemma}
\begin{proof}
(1) Since $P_{t}$ is uniserial and has the unique composition series
\[
S_{t} \quad S_{a_{1}} \quad S_{a_{2}} \quad \cdots \quad S_{a_{p}} \quad S_{t},
\]
where the left end is the top of $P_{t}$ and the right end is the socle of $P_{t}$. So we have $\mathrm{soc}(P_{t}/S_{t}) = S_{a_{p}}$. Thus $E(P_{t}/S_{t}) \cong P_{a_{p}}$ and the assertion follows. \\
\indent
(2) Since $\mathrm{rad}(P_{t}) / \mathrm{soc}(P_{t})$ is the direct sum of two uniserial modules, $P_{t}$ has the composition series
\[
\begin{array}{cccccc}
 & S_{a_{1}} & S_{a_{2}} & \cdots & S_{a_{p}} & \\
S_{t} & & & & & S_{t},  \\
 & S_{b_{1}} & S_{b_{2}} & \cdots & S_{b_{r}} & 
\end{array}
\]
where the left end is the top of $P_{t}$ and the right end is the socle of $P_{t}$. Thus $E(P_{t}/S_{t}) \cong P_{a_{p}} \oplus P_{b_{r}}$ and the assertion follows. \\
\indent
(3) This follows by (1) and (2).   
\end{proof}

Take a minimal injective presentation $0 \to S_{t} \to E^{0}_{t} \overset{f}{\to} E^{1}_{t}$ and define a complex $E^{\bullet}_{t}$ as the mapping cone of $f:E^{0}_{t} \to E^{1}_{t}$. Set 
\[
T^{\bullet}_{t} = T_{1} \oplus E^{\bullet}_{t} \quad \text{with} \quad T_{1} = \bigoplus_{i \in \Lambda_{0} \setminus \{ t \} } P_{i}. 
\]
Then by Theorem \ref{tilting_complex}, Remark \ref{nakayama_functor} and Lemma \ref{self_ext}(3) $T^{\bullet}_{t}$ is a tilting complex and $\mathrm{End}_{\mathscr{K}(\mathrm{mod}\text{-}\Lambda)}(T^{\bullet}_{t})$ is the reflection of $\Lambda$ at $t$. Set $\Gamma = \mathrm{End}_{\mathscr{K}(\mathrm{mod}\text{-}\Lambda)}(T^{\bullet}_{t})$ and let $(\Gamma_{0}, \Gamma_{1}, \sigma)$ be the quiver with relations of $\Gamma$. According to Proposition \ref{der_eq_Brauer_tree}, we have the following. 

\begin{remark}\label{proj_Brauer}
For $x, y \in \Gamma_{0}$, $\mathrm{Hom}_{\Gamma}(P_{x}, P_{y}) \ne 0$ if and only if $x, y$ belong to the same cycle in $\Gamma$. If this is the case, $\mathrm{dim}_{K} \mathrm{Hom}_{\Gamma}(P_{x}, P_{y}) = 1$ for $x \ne y$ and $\mathrm{dim}_{K} \mathrm{End}_{\Gamma}(P_{x}) = 2$. 
\end{remark}

Note that $\Gamma_{0} = (\Lambda_{0} \setminus \{ t \} ) \cup \{t^{\prime} \}$, where $t^{\prime}$ is the vertex corresponding to $E^{\bullet}_{t}$. Since $\Gamma$ is a Brauer tree algebra, the relations $\sigma$ is determined automatically by $\Gamma_{0}$ and $\Gamma_{1}$. To determine $\Gamma_{1}$, by Remark \ref{proj_Brauer} it suffices to consider the following cycles in $(\Lambda_{0}, \Lambda_{1}, \rho)$. We denote by the diagram 
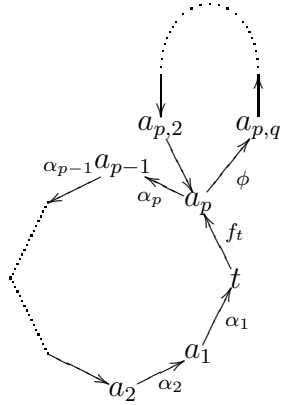
\[
\begin{xy}
(15,-15) *{t}="t",
(10,-25) *{a_{1}}="a(1)",
(0,-30) *{a_{2}}="a(2)",
(-10,-25) *{}="a(3)",
(-15,-15) *{}="a(i)",
(-10,-5) *{}="a(p-2)", 
(0,0) *{a_{p-1}}="a(p-1)",
(10,-5) *{a_{p}}="a(p)", 
(18,5) *{a_{p,q}}="a(pq)",
(18,12) *{}="a(pq-1)",
(5,12) *{}="a(p3)",
(5,5) *{a_{p,2}}="a(p2)", 
\ar_{\alpha_{1}} "a(1)" ; "t"  
\ar_{\alpha_{2}} "a(2)" ; "a(1)" 
\ar "a(3)" ; "a(2)" 
\ar @{.} "a(i)" ; "a(3)"  
\ar @{.} "a(p-2)" ; "a(i)" 
\ar_{\alpha_{p-1}} "a(p-1)" ; "a(p-2)"  
\ar^{\alpha_{p}} "a(p)" ; "a(p-1)" 
\ar_{f_{t}} "t" ; "a(p)" 
\ar "a(p2)" ; "a(p)" 
\ar "a(p3)" ; "a(p2)" 
\ar @{.} @(u,u) "a(pq-1)" ; "a(p3)" 
\ar "a(pq)" ; "a(pq-1)" 
\ar_{\phi} "a(p)" ; "a(pq)" 
\end{xy}
\]
with $p, q \ge 1$ the situation that $t$ belongs to at most one cycle and set $a_{p,1}=a_{p}$. Also, we denote by the diagram 
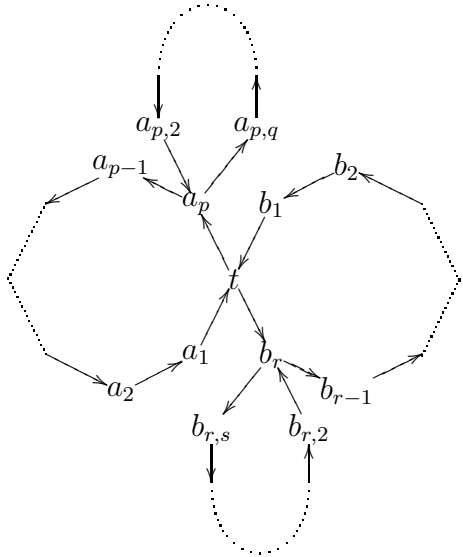
\[
\begin{xy}
(15,-15) *{t}="t",
(10,-25) *{a_{1}}="a(1)",
(0,-30) *{a_{2}}="a(2)",
(-10,-25) *{}="a(3)",
(-15,-15) *{}="a(i)",
(-10,-5) *{}="a(p-2)", 
(0,0) *{a_{p-1}}="a(p-1)",
(10,-5) *{a_{p}}="a(p)", 
(18,5) *{a_{p,q}}="a(pq)",
(18,12) *{}="a(pq-1)",
(5,12) *{}="a(p3)",
(5,5) *{a_{p,2}}="a(p2)", 
(20,-5) *{b_{1}}="b(1)",
(30,0) *{b_{2}}="b(2)",
(40,-5) *{}="b(3)",
(45,-15) *{}="b(i)",
(40,-25) *{}="b(r-2)",
(30,-30) *{b_{r-1}}="b(r-1)",
(20,-25) *{b_{r}}="b(r)",
(12,-35) *{b_{r,s}}="b(rs)",
(12,-42) *{}="b(rs-1)",
(25,-42) *{}="b(r3)",
(25,-35) *{b_{r,2}}="b(r2)",

\ar "a(1)" ; "t"  
\ar "a(2)" ; "a(1)" 
\ar "a(3)" ; "a(2)" 
\ar @{.} "a(i)" ; "a(3)"  
\ar @{.} "a(p-2)" ; "a(i)" 
\ar "a(p-1)" ; "a(p-2)"  
\ar "a(p)" ; "a(p-1)" 
\ar "t" ; "a(p)" 
\ar "a(p2)" ; "a(p)" 
\ar "a(p3)" ; "a(p2)" 
\ar @{.} @(u,u) "a(pq-1)" ; "a(p3)" 
\ar "a(pq)" ; "a(pq-1)" 
\ar "a(p)" ; "a(pq)" 
\ar "b(1)" ; "t" 
\ar "b(2)" ; "b(1)" 
\ar "b(3)" ; "b(2)" 
\ar @{.} "b(i)" ; "b(3)" 
\ar @{.} "b(r-2)" ; "b(i)" 
\ar "b(r-1)" ; "b(r-2)" 
\ar "b(r)" ; "b(r-1)" 
\ar "t" ; "b(r)" 
\ar "b(r2)" ; "b(r)" 
\ar "b(r3)" ; "b(r2)" 
\ar @{.} @(d,d) "b(rs-1)" ; "b(r3)" 
\ar "b(rs)" ; "b(rs-1)" 
\ar "b(r)" ; "b(rs)" 

\end{xy}
\]
with $p, q, r, s \ge 1$ the situation that $t$ belongs to two cycles and set $a_{p,1}=a_{p}$ and $b_{r,1}=b_{r}$. 

\begin{lemma}\label{a_cycle}
If $t$ belongs to at most one cycle, then the following hold. 
\begin{enumerate}
\item[(1)] There exists $\zeta_{a_{p}} \in \mathrm{Hom}_{\mathscr{K}(\mathrm{mod}\text{-}\Lambda)}(P_{a_{p}}, E^{\bullet}_{t})$ with $\zeta_{a_{p}} \in \mathrm{rad}(\Gamma) \setminus \mathrm{rad}^{2}(\Gamma)$. 
\item[(2)] There exists $\eta_{a_{p,q}} \in \mathrm{Hom}_{\mathscr{K}(\mathrm{mod}\text{-}\Lambda)}(E^{\bullet}_{t}, P_{a_{p,q}})$ with $\eta_{a_{p,q}} \in \mathrm{rad}(\Gamma) \setminus \mathrm{rad}^{2}(\Gamma)$. 
\item[(3)] There exists $\theta_{a_{p}} \in \mathrm{Hom}_{\mathscr{K}(\mathrm{mod}\text{-}\Lambda)}(P_{a_{1}}, P_{a_{p}})$ with $\theta_{a_{p}} \in \mathrm{rad}(\Gamma) \setminus \mathrm{rad}^{2}(\Gamma)$. 
\end{enumerate}
\end{lemma}
\begin{proof}
By Lemma \ref{self_ext}(1), the complex $E^{\bullet}_{t}$ is of the form  
\[
\xymatrix{
E^{\bullet}_{t} : \cdots \ar [r] & 0 \ar [r] & P_{t} \ar [r] & P_{a_{p}} \ar [r] & 0 \ar [r] & \cdots,
}
\]
where $P_{a_{p}}$ is the $0$th term. \\
\indent
(1) We have a cochain map 
\[
\xymatrix{
P_{a_{p}} \ar [d]<-10pt>_{\zeta_{a_{p}}} : \cdots \ar [r] & 0 \ar [r] \ar [d] & 0 \ar [r] \ar [d] & P_{a_{p}} \ar [r] \ar [d]^{\mathrm{id}} & 0 \ar [r] \ar [d] & \cdots \\
E^{\bullet}_{t} : \cdots \ar [r] & 0 \ar [r] & P_{t} \ar [r]_{f_{t}} & P_{a_{p}} \ar [r] & 0 \ar [r] & \cdots ,
}
\]
which is obviously not homotopic to zero. Since $P_{a_{p}}$ is indecomposable and injective, $\zeta_{a_{p}} \in \mathrm{rad}(\Gamma) \setminus \mathrm{rad}^{2}(\Gamma)$. \\
\indent
(2) Consider first the case where $q \ne 1$. Since $\phi f_{t} = 0$, we have a cochain map 
\[
\xymatrix{
E^{\bullet}_{t} \ar [d]<-10pt>_{\eta_{a_{p,q}}} : \cdots \ar [r] & 0 \ar [r] \ar [d] & P_{t} \ar [r]^{f_{t}} \ar [d] & P_{a_{p}} \ar [r] \ar [d]^{\phi} & 0 \ar [r] \ar [d] & \cdots \\
P_{a_{p,q}} : \cdots \ar [r] & 0 \ar [r] & 0 \ar [r] & P_{a_{p,q}} \ar [r] & 0 \ar [r] & \cdots 
}
\]
which is obviously not homotopic to zero. Since $\phi \in \mathrm{rad}(\Lambda) \setminus \mathrm{rad}^{2}(\Lambda)$, we have $\eta_{a_{p,q}} \in \mathrm{rad}(\Gamma) \setminus \mathrm{rad}^{2}(\Gamma)$. Next, assume that $q = 1$. We have a cochain map 
\[
\xymatrix{
E^{\bullet}_{t} \ar [d]<-10pt>_{\eta_{a_{p}}} : \cdots \ar [r] & 0 \ar [r] \ar [d] & P_{t} \ar [r]^{f_{t}} \ar [d] & P_{a_{p}} \ar [r] \ar [d]^{f_{t} \alpha_{1} \cdots \alpha_{p}} & 0 \ar [r] \ar [d] & \cdots \\
P_{a_{p}} : \cdots \ar [r] & 0 \ar [r] & 0 \ar [r] & P_{a_{p}} \ar [r] & 0 \ar [r] & \cdots 
}
\]
which is obviously not homotopic to zero. For any $i \ne p$, $\mathrm{Hom}_{\Lambda}(\mathrm{Cok}\ f_{t}, P_{a_{i}})=0$ and hence $\mathrm{Hom}_{\mathscr{K}(\mathrm{mod}\text{-}\Lambda)}(E^{\bullet}_{t}, P_{a_{i}})=0$. It follows that $\eta_{a_{p}} \in \mathrm{rad}(\Gamma) \setminus \mathrm{rad}^{2}(\Gamma)$. \\
\indent
(3) We have $0 \ne f \alpha_{1} \in \mathrm{Hom}_{\Lambda}(P_{a_{1}},P_{a_{p}})$, which yields a nonzero map $\theta_{a_{p}} \in \mathrm{Hom}_{\mathscr{K}(\mathrm{mod}\text{-}\Lambda)}(P_{a_{1}}, P_{a_{p}})$. For any $i \ne 1$, since $f \alpha_{1}$ does not factor through $P_{a_{i}}$, $\theta_{a_{p}}$ does not factor through $P_{a_{i}}$. Also, for any $i \ne p$, $\mathrm{Hom}_{\Lambda}(P_{a_{i}}, \mathrm{Cok}\ f_{t})=0$ and $\mathrm{Hom}_{\mathscr{K}(\mathrm{mod}\text{-}\Lambda)}(P_{a_{i}},E^{\bullet}_{t})=0$. It follows that $\theta_{a_{p}} \in \mathrm{rad}(\Gamma) \setminus \mathrm{rad}^{2}(\Gamma)$. 
\end{proof}

\begin{lemma}\label{two_cycles}
If $t$ belongs to two cycles, then the following hold. 
\begin{enumerate}
\item[(1)] There exist $\zeta_{a_{p}} \in \mathrm{Hom}_{\mathscr{K}(\mathrm{mod}\text{-}\Lambda)}(P_{a_{p}}, E^{\bullet}_{t})$ with $\zeta_{a_{p}} \in \mathrm{rad}(\Gamma) \setminus \mathrm{rad}^{2}(\Gamma)$ and $\zeta_{b_{r}} \in \mathrm{Hom}_{\mathscr{K}(\mathrm{mod}\text{-}\Lambda)}(P_{b_{r}}, E^{\bullet}_{t})$ with $\zeta_{b_{r}} \in \mathrm{rad}(\Gamma) \setminus \mathrm{rad}^{2}(\Gamma)$. 
\item[(2)] There exist $\eta_{a_{p,q}} \in \mathrm{Hom}_{\mathscr{K}(\mathrm{mod}\text{-}\Lambda)}(E^{\bullet}_{t}, P_{a_{p,q}})$ with $\eta_{a_{p,q}} \in \mathrm{rad}(\Gamma) \setminus \mathrm{rad}^{2}(\Gamma)$ and $\eta_{b_{r,s}} \in \mathrm{Hom}_{\mathscr{K}(\mathrm{mod}\text{-}\Lambda)}(E^{\bullet}_{t}, P_{b_{r,s}})$ with $\eta_{b_{r,s}} \in \mathrm{rad}(\Gamma) \setminus \mathrm{rad}^{2}(\Gamma)$. 
\item[(3)] There exist $\theta_{a_{p}} \in \mathrm{Hom}_{\mathscr{K}(\mathrm{mod}\text{-}\Lambda)}(P_{a_{1}}, P_{a_{p}})$ with $\theta_{a_{p}} \in \mathrm{rad}(\Gamma) \setminus \mathrm{rad}^{2}(\Gamma)$ and $\theta_{b_{r}} \in \mathrm{Hom}_{\mathscr{K}(\mathrm{mod}\text{-}\Lambda)}(P_{b_{1}}, P_{b_{r}})$ with $\theta_{b_{r}} \in \mathrm{rad}(\Gamma) \setminus \mathrm{rad}^{2}(\Gamma)$. 
\end{enumerate}
\end{lemma}
\begin{proof}
By Lemma \ref{self_ext}(2), the complex $E^{\bullet}_{t}$ is of the form  
\[
\xymatrix{
E^{\bullet}_{t} : \cdots \ar [r] & 0 \ar [r] & P_{t} \ar [r] & P_{a_{p}} \oplus P_{b_{r}} \ar [r] & 0 \ar [r] & \cdots, 
}
\]
where $P_{a_{p}} \oplus P_{b_{r}}$ is the $0$th term. The assertions follow by the same arguments as in the proof of Lemma \ref{a_cycle}. 
\end{proof}

According to Lemmas \ref{a_cycle} and \ref{two_cycles}, we have the following new arrows in $\Gamma_{1}$. We denote by $\xybox{\ar @{=>} (10,0)}$ the arrows defined by $\zeta_{*}$, by $\xybox{\ar @{~>} (10,0)}$ the arrows defined by $\eta_{*}$ and by $\xybox{\ar @{-->} (10,0)}$ the arrows defined by $\theta_{*}$. In the next theorem, the left hand side diagrams denote cycles in $(\Lambda_{0}, \Lambda_{1}, \rho)$ and the right hand side diagrams denote cycles in $(\Gamma_{0}, \Gamma_{1}, \sigma)$. 

\begin{theorem}\label{reflection_quiver}
The following hold. 
\begin{enumerate}
\item[(1)] If $t$ and $a_{p}$ belong to at most one cycle, then the reflection for $\Lambda$ at $t$ gives rise to the following transformation:
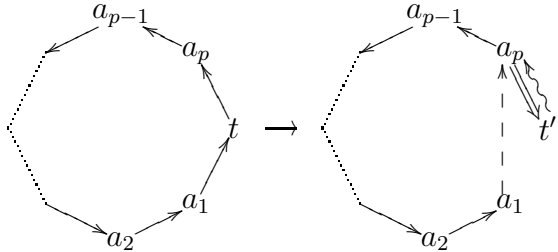
\[
\begin{array}{ccc}
\begin{xy}
(15,-15) *{t}="t",
(10,-25) *{a_{1}}="a(1)",
(0,-30) *{a_{2}}="a(2)",
(-10,-25) *{}="a(3)",
(-15,-15) *{}="a(i)",
(-10,-5) *{}="a(p-2)", 
(0,0) *{a_{p-1}}="a(p-1)",
(10,-5) *{a_{p}}="a(p)", 
\ar "a(1)" ; "t" 
\ar "a(2)" ; "a(1)" 
\ar "a(3)" ; "a(2)" 
\ar @{.} "a(i)" ; "a(3)" 
\ar @{.} "a(p-2)" ; "a(i)" 
\ar "a(p-1)" ; "a(p-2)" 
\ar "a(p)" ; "a(p-1)" 
\ar "t" ; "a(p)" 
\end{xy}
&
\begin{xy}
(-2,-15) *{}="a",
(2,-15) *{}="b",
\SelectTips{eu}{12} \ar @{->} "a" ; "b"
\end{xy}
&
\begin{xy}
(15,-15) *{t^{\prime}}="t",
(10,-25) *{a_{1}}="a(1)",
(0,-30) *{a_{2}}="a(2)",
(-10,-25) *{}="a(3)",
(-15,-15) *{}="a(i)",
(-10,-5) *{}="a(p-2)", 
(0,0) *{a_{p-1}}="a(p-1)",
(10,-5) *{a_{p}}="a(p)", 
\ar @{=>} @<-1.5pt> "a(p)" ; "t" 
\ar @{~>} @<-3.5pt> "t" ; "a(p)" 
\ar @{-->} @<3pt> "a(1)" ; "a(p)" 
\ar "a(2)" ; "a(1)" 
\ar "a(3)" ; "a(2)" 
\ar @{.} "a(i)" ; "a(3)" 
\ar @{.} "a(p-2)" ; "a(i)" 
\ar "a(p-1)" ; "a(p-2)" 
\ar "a(p)" ; "a(p-1)" 
\end{xy}
\end{array}
\]
\item[(2)] If $t$ belongs to at most one cycle and $a_{p}$ belongs to two cycles, then the reflection for $\Lambda$ at $t$ gives rise to the following transformation:
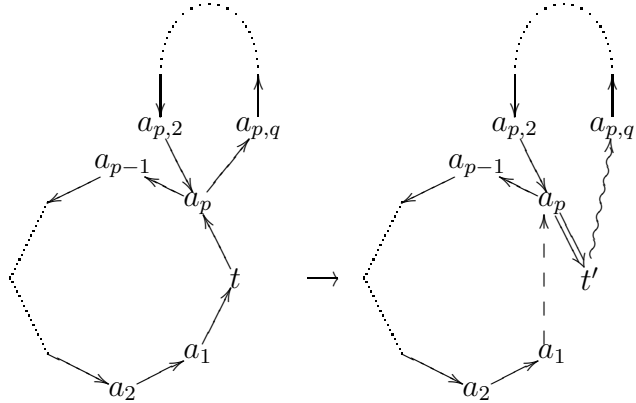
\[
\begin{array}{ccc}
\begin{xy}
(15,-15) *{t}="t",
(10,-25) *{a_{1}}="a(1)",
(0,-30) *{a_{2}}="a(2)",
(-10,-25) *{}="a(3)",
(-15,-15) *{}="a(i)",
(-10,-5) *{}="a(p-2)", 
(0,0) *{a_{p-1}}="a(p-1)",
(10,-5) *{a_{p}}="a(p)", 
(18,5) *{a_{p,q}}="a(pq)",
(18,12) *{}="a(pq-1)",
(5,12) *{}="a(p3)",
(5,5) *{a_{p,2}}="a(p2)", 
\ar "a(1)" ; "t"  
\ar "a(2)" ; "a(1)" 
\ar "a(3)" ; "a(2)" 
\ar @{.} "a(i)" ; "a(3)" 
\ar @{.} "a(p-2)" ; "a(i)" 
\ar "a(p-1)" ; "a(p-2)" 
\ar "a(p)" ; "a(p-1)" 
\ar "t" ; "a(p)" 
\ar "a(p2)" ; "a(p)" 
\ar "a(p3)" ; "a(p2)" 
\ar @{.} @(u,u) "a(pq-1)" ; "a(p3)" 
\ar "a(pq)" ; "a(pq-1)" 
\ar "a(p)" ; "a(pq)" 
\end{xy}
&
\begin{xy}
(-2,-15) *{}="a",
(2,-15) *{}="b",
\SelectTips{eu}{12} \ar @{->} "a" ; "b"
\end{xy}
&
\begin{xy}
(15,-15) *{t^{\prime}}="t",
(10,-25) *{a_{1}}="a(1)",
(0,-30) *{a_{2}}="a(2)",
(-10,-25) *{}="a(3)",
(-15,-15) *{}="a(i)",
(-10,-5) *{}="a(p-2)", 
(0,0) *{a_{p-1}}="a(p-1)",
(10,-5) *{a_{p}}="a(p)", 
(18,5) *{a_{p,q}}="a(pq)",
(18,12) *{}="a(pq-1)",
(5,12) *{}="a(p3)",
(5,5) *{a_{p,2}}="a(p2)", 
\ar @{=>}  "a(p)" ; "t" 
\ar @{~>}  "t" ; "a(pq)" 
\ar @{-->} @<3pt> "a(1)" ; "a(p)" 
\ar "a(2)" ; "a(1)" 
\ar "a(3)" ; "a(2)" 
\ar @{.} "a(i)" ; "a(3)" 
\ar @{.} "a(p-2)" ; "a(i)" 
\ar "a(p-1)" ; "a(p-2)" 
\ar "a(p)" ; "a(p-1)" 
\ar "a(p2)" ; "a(p)" 
\ar "a(p3)" ; "a(p2)" 
\ar @{.} @(u,u) "a(pq-1)" ; "a(p3)" 
\ar "a(pq)" ; "a(pq-1)" 
\end{xy}
\end{array}
\]
\item[(3)] If $t$ belongs to two cycles and $a_{p}$ and $b_{r}$ belong to at most one cycle, then the reflection for $\Lambda$ at $t$ gives rise to the following transformation:
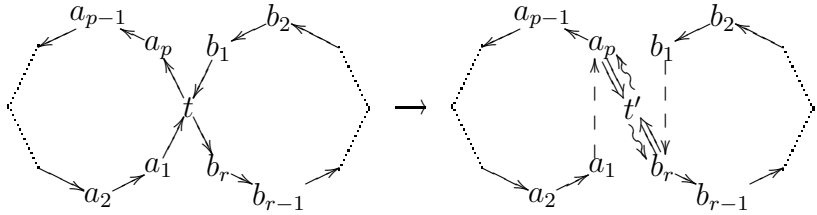
\[
\begin{array}{ccc}
\begin{xy}
(12,-12) *{t}="t",
(8,-20) *{a_{1}}="a(1)",
(0,-24) *{a_{2}}="a(2)",
(-8,-20) *{}="a(3)",
(-12,-12) *{}="a(i)",
(-8,-4) *{}="a(p-2)", 
(0,0) *{a_{p-1}}="a(p-1)",
(8,-4) *{a_{p}}="a(p)", 
(16,-4) *{b_{1}}="b(1)",
(24,0) *{b_{2}}="b(2)",
(32,-4) *{}="b(3)",
(36,-12) *{}="b(i)",
(32,-20) *{}="b(r-2)",
(24,-24) *{b_{r-1}}="b(r-1)",
(16,-20) *{b_{r}}="b(r)",
\ar "a(1)" ; "t"  
\ar "a(2)" ; "a(1)" 
\ar "a(3)" ; "a(2)" 
\ar @{.} "a(i)" ; "a(3)" 
\ar @{.} "a(p-2)" ; "a(i)" 
\ar "a(p-1)" ; "a(p-2)" 
\ar "a(p)" ; "a(p-1)" 
\ar "t" ; "a(p)" 
\ar "b(1)" ; "t" 
\ar "b(2)" ; "b(1)" 
\ar "b(3)" ; "b(2)" 
\ar @{.} "b(i)" ; "b(3)" 
\ar @{.} "b(r-2)" ; "b(i)" 
\ar "b(r-1)" ; "b(r-2)" 
\ar "b(r)" ; "b(r-1)" 
\ar "t" ; "b(r)" 
\end{xy}
&
\begin{xy}
(-2,-12) *{}="a",
(2,-12) *{}="b",
\SelectTips{eu}{12} \ar @{->} "a" ; "b"
\end{xy}
&
\begin{xy}
(12,-12) *{t^{\prime}}="t",
(8,-20) *{a_{1}}="a(1)",
(0,-24) *{a_{2}}="a(2)",
(-8,-20) *{}="a(3)",
(-12,-12) *{}="a(i)",
(-8,-4) *{}="a(p-2)", 
(0,0) *{a_{p-1}}="a(p-1)",
(8,-4) *{a_{p}}="a(p)", 
(16,-4) *{b_{1}}="b(1)",
(24,0) *{b_{2}}="b(2)",
(32,-4) *{}="b(3)",
(36,-12) *{}="b(i)",
(32,-20) *{}="b(r-2)",
(24,-24) *{b_{r-1}}="b(r-1)",
(16,-20) *{b_{r}}="b(r)",
\ar @{=>} @<-1.5pt> "a(p)" ; "t" 
\ar @{~>} @<-3.5pt> "t" ; "a(p)" 
\ar @{-->} @<3pt> "a(1)" ; "a(p)" 
\ar "a(2)" ; "a(1)" 
\ar "a(3)" ; "a(2)" 
\ar @{.} "a(i)" ; "a(3)" 
\ar @{.} "a(p-2)" ; "a(i)" 
\ar "a(p-1)" ; "a(p-2)" 
\ar "a(p)" ; "a(p-1)" 
\ar @{=>} @<-1pt> "b(r)" ; "t" 
\ar @{~>} @<-4pt> "t" ; "b(r)" 
\ar @{-->} @<1pt> "b(1)" ; "b(r)" 
\ar "b(2)" ; "b(1)" 
\ar "b(3)" ; "b(2)" 
\ar @{.} "b(i)" ; "b(3)" 
\ar @{.} "b(r-2)" ; "b(i)" 
\ar "b(r-1)" ; "b(r-2)" 
\ar "b(r)" ; "b(r-1)" 
\end{xy}
\end{array}
\]
\item[(4)] If $t$ and $a_{p}$ belong to two cycles and $b_{r}$ belongs to at most one cycle, then the reflection for $\Lambda$ at $t$ gives rise to the following transformation:
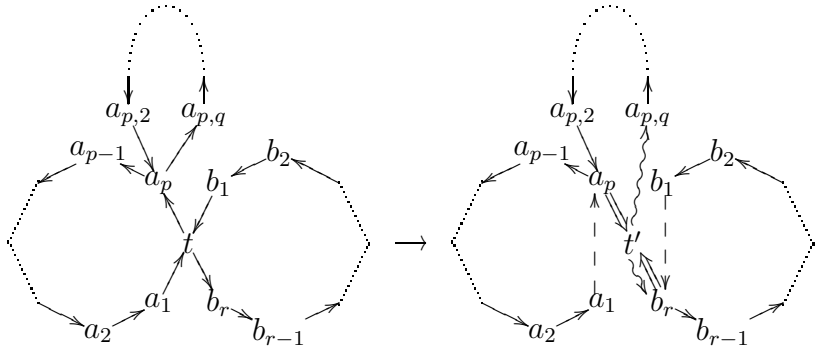
\[
\begin{array}{ccc}
\begin{xy}
(12,-12) *{t}="t",
(8,-20) *{a_{1}}="a(1)",
(0,-24) *{a_{2}}="a(2)",
(-8,-20) *{}="a(3)",
(-12,-12) *{}="a(i)",
(-8,-4) *{}="a(p-2)", 
(0,0) *{a_{p-1}}="a(p-1)",
(8,-4) *{a_{p}}="a(p)", 
(14,5) *{a_{p,q}}="a(pq)",
(14,10) *{}="a(pq-1)",
(4,10) *{}="a(p3)",
(4,5) *{a_{p,2}}="a(p2)", 
(16,-4) *{b_{1}}="b(1)",
(24,0) *{b_{2}}="b(2)",
(32,-4) *{}="b(3)",
(36,-12) *{}="b(i)",
(32,-20) *{}="b(r-2)",
(24,-24) *{b_{r-1}}="b(r-1)",
(16,-20) *{b_{r}}="b(r)",
\ar "a(1)" ; "t" 
\ar "a(2)" ; "a(1)" 
\ar "a(3)" ; "a(2)" 
\ar @{.} "a(i)" ; "a(3)" 
\ar @{.} "a(p-2)" ; "a(i)" 
\ar "a(p-1)" ; "a(p-2)" 
\ar "a(p)" ; "a(p-1)" 
\ar "t" ; "a(p)" 
\ar "a(p2)" ; "a(p)" 
\ar "a(p3)" ; "a(p2)" 
\ar @{.} @(u,u) "a(pq-1)" ; "a(p3)" 
\ar "a(pq)" ; "a(pq-1)" 
\ar "a(p)" ; "a(pq)" 
\ar "b(1)" ; "t" 
\ar "b(2)" ; "b(1)" 
\ar "b(3)" ; "b(2)" 
\ar @{.} "b(i)" ; "b(3)" 
\ar @{.} "b(r-2)" ; "b(i)" 
\ar "b(r-1)" ; "b(r-2)" 
\ar "b(r)" ; "b(r-1)" 
\ar "t" ; "b(r)" 
\end{xy}
&
\begin{xy}
(-2,-12) *{}="a",
(2,-12) *{}="b",
\SelectTips{eu}{12} \ar @{->} "a" ; "b"
\end{xy}
&
\begin{xy}
(12,-12) *{t^{\prime}}="t",
(8,-20) *{a_{1}}="a(1)",
(0,-24) *{a_{2}}="a(2)",
(-8,-20) *{}="a(3)",
(-12,-12) *{}="a(i)",
(-8,-4) *{}="a(p-2)", 
(0,0) *{a_{p-1}}="a(p-1)",
(8,-4) *{a_{p}}="a(p)", 
(14,5) *{a_{p,q}}="a(pq)",
(14,10) *{}="a(pq-1)",
(4,10) *{}="a(p3)",
(4,5) *{a_{p,2}}="a(p2)", 
(16,-4) *{b_{1}}="b(1)",
(24,0) *{b_{2}}="b(2)",
(32,-4) *{}="b(3)",
(36,-12) *{}="b(i)",
(32,-20) *{}="b(r-2)",
(24,-24) *{b_{r-1}}="b(r-1)",
(16,-20) *{b_{r}}="b(r)",
\ar @{=>}  "a(p)" ; "t" 
\ar @{~>}  "t"; "a(pq)" 
\ar @{-->} @<3pt> "a(1)" ; "a(p)" 
\ar "a(2)" ; "a(1)" 
\ar "a(3)" ; "a(2)" 
\ar @{.} "a(i)" ; "a(3)" 
\ar @{.} "a(p-2)" ; "a(i)" 
\ar "a(p-1)" ; "a(p-2)" 
\ar "a(p)" ; "a(p-1)" 
\ar "a(p2)" ; "a(p)" 
\ar "a(p3)" ; "a(p2)" 
\ar @{.} @(u,u) "a(pq-1)" ; "a(p3)" 
\ar "a(pq)" ; "a(pq-1)" 
\ar @{=>} @<-1pt> "b(r)" ; "t" 
\ar @{~>} @<-4pt> "t" ; "b(r)" 
\ar @{-->} @<1pt> "b(1)" ; "b(r)" 
\ar "b(2)" ; "b(1)" 
\ar "b(3)" ; "b(2)" 
\ar @{.} "b(i)" ; "b(3)" 
\ar @{.} "b(r-2)" ; "b(i)" 
\ar "b(r-1)" ; "b(r-2)" 
\ar "b(r)" ; "b(r-1)" 
\end{xy}
\end{array}
\]
\item[(5)] If $t$, $a_{p}$ and $b_{r}$ belong to two cycles, then the reflection for $\Lambda$ at $t$ gives rise to the following transformation:
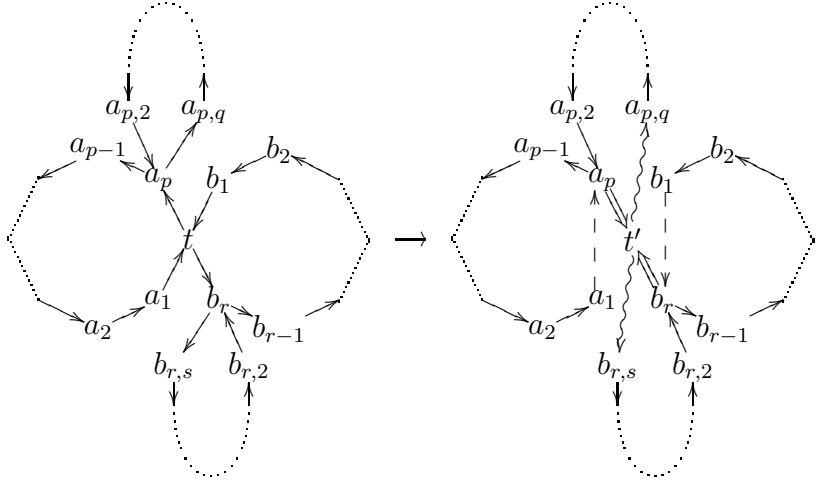
\[
\begin{array}{ccc}
\begin{xy}
(12,-12) *{t}="t",
(8,-20) *{a_{1}}="a(1)",
(0,-24) *{a_{2}}="a(2)",
(-8,-20) *{}="a(3)",
(-12,-12) *{}="a(i)",
(-8,-4) *{}="a(p-2)", 
(0,0) *{a_{p-1}}="a(p-1)",
(8,-4) *{a_{p}}="a(p)", 
(14,5) *{a_{p,q}}="a(pq)",
(14,10) *{}="a(pq-1)",
(4,10) *{}="a(p3)",
(4,5) *{a_{p,2}}="a(p2)", 
(16,-4) *{b_{1}}="b(1)",
(24,0) *{b_{2}}="b(2)",
(32,-4) *{}="b(3)",
(36,-12) *{}="b(i)",
(32,-20) *{}="b(r-2)",
(24,-24) *{b_{r-1}}="b(r-1)",
(16,-20) *{b_{r}}="b(r)",
(10,-29) *{b_{r,s}}="b(rs)",
(10,-34) *{}="b(rs-1)",
(20,-34) *{}="b(r3)",
(20,-29) *{b_{r,2}}="b(r2)",
\ar "a(1)" ; "t" 
\ar "a(2)" ; "a(1)" 
\ar "a(3)" ; "a(2)" 
\ar @{.} "a(i)" ; "a(3)" 
\ar @{.} "a(p-2)" ; "a(i)" 
\ar "a(p-1)" ; "a(p-2)" 
\ar "a(p)" ; "a(p-1)" 
\ar "t" ; "a(p)" 
\ar "a(p2)" ; "a(p)" 
\ar "a(p3)" ; "a(p2)" 
\ar @{.} @(u,u) "a(pq-1)" ; "a(p3)" 
\ar "a(pq)" ; "a(pq-1)" 
\ar "a(p)" ; "a(pq)" 
\ar "b(1)" ; "t" 
\ar "b(2)" ; "b(1)" 
\ar "b(3)" ; "b(2)" 
\ar @{.} "b(i)" ; "b(3)" 
\ar @{.} "b(r-2)" ; "b(i)" 
\ar "b(r-1)" ; "b(r-2)" 
\ar "b(r)" ; "b(r-1)" 
\ar "t" ; "b(r)" 
\ar "b(r2)" ; "b(r)" 
\ar "b(r3)" ; "b(r2)" 
\ar @{.} @(d,d) "b(rs-1)" ; "b(r3)" 
\ar "b(rs)" ; "b(rs-1)" 
\ar "b(r)" ; "b(rs)" 
\end{xy}
&
\begin{xy}
(-2,-12) *{}="a",
(2,-12) *{}="b",
\SelectTips{eu}{12} \ar @{->} "a" ; "b"
\end{xy}
&
\begin{xy}
(12,-12) *{t^{\prime}}="t",
(8,-20) *{a_{1}}="a(1)",
(0,-24) *{a_{2}}="a(2)",
(-8,-20) *{}="a(3)",
(-12,-12) *{}="a(i)",
(-8,-4) *{}="a(p-2)", 
(0,0) *{a_{p-1}}="a(p-1)",
(8,-4) *{a_{p}}="a(p)", 
(14,5) *{a_{p,q}}="a(pq)",
(14,10) *{}="a(pq-1)",
(4,10) *{}="a(p3)",
(4,5) *{a_{p,2}}="a(p2)", 
(16,-4) *{b_{1}}="b(1)",
(24,0) *{b_{2}}="b(2)",
(32,-4) *{}="b(3)",
(36,-12) *{}="b(i)",
(32,-20) *{}="b(r-2)",
(24,-24) *{b_{r-1}}="b(r-1)",
(16,-20) *{b_{r}}="b(r)",
(10,-29) *{b_{r,s}}="b(rs)",
(10,-34) *{}="b(rs-1)",
(20,-34) *{}="b(r3)",
(20,-29) *{b_{r,2}}="b(r2)",
\ar @{=>}  "a(p)" ; "t" 
\ar @{~>}  "t" ; "a(pq)" 
\ar @{-->} @<3pt> "a(1)" ; "a(p)" 
\ar "a(2)" ; "a(1)" 
\ar "a(3)" ; "a(2)" 
\ar @{.} "a(i)" ; "a(3)" 
\ar @{.} "a(p-2)" ; "a(i)" 
\ar "a(p-1)" ; "a(p-2)" 
\ar "a(p)" ; "a(p-1)" 
\ar "a(p2)" ; "a(p)" 
\ar "a(p3)" ; "a(p2)" 
\ar @{.} @(u,u) "a(pq-1)" ; "a(p3)" 
\ar "a(pq)" ; "a(pq-1)" 
\ar @{=>}  "b(r)" ; "t" 
\ar @{~>}  "t" ; "b(rs)" 
\ar @{-->} @<1pt> "b(1)" ; "b(r)" 
\ar "b(2)" ; "b(1)" 
\ar "b(3)" ; "b(2)" 
\ar @{.} "b(i)" ; "b(3)" 
\ar @{.} "b(r-2)" ; "b(i)" 
\ar "b(r-1)" ; "b(r-2)" 
\ar "b(r)" ; "b(r-1)" 
\ar "b(r2)" ; "b(r)" 
\ar "b(r3)" ; "b(r2)" 
\ar @{.} @(d,d) "b(rs-1)" ; "b(r3)" 
\ar "b(rs)" ; "b(rs-1)" 
\end{xy}
\end{array}
\]
\end{enumerate}

\end{theorem}
\begin{proof}
(1), (2) follow by Lemma \ref{a_cycle} and (3), (4), (5) follow by Lemma \ref{two_cycles}. 
\end{proof}

Let $\Lambda$ be determined by a Brauer tree $B$ whose edges are identified with the vertices of $(\Lambda_{0}, \Lambda_{1}, \rho)$. We will describe a way to transform $B$ into a Brauer tree $B^{\prime}$ determining $\Gamma$. Consider first the case where $t$ is an end edge of $B$:
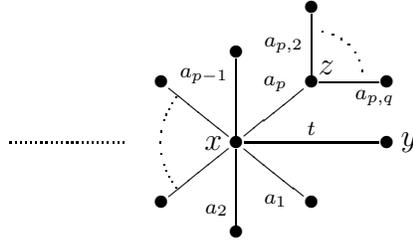
\[
\begin{xy}
(-15,0) *{}="x",
(-30,0) *{}="xx",
(0,0) *{\bullet}="a",
(20,0) *{\bullet}="b",
(10,8) *{\bullet}="i",
(20,8) *{\bullet}="ii",
(10,18) *{\bullet}="iii",
(17,8) *{}="i-i",
(10,15) *{}="i-ii",
(0,12) *{\bullet}="j",
(-10,8) *{\bullet}="k",
(-8,6) *{}="kk",
(-10,-8) *{\bullet}="l",
(-8,-6) *{}="ll",
(0,-12) *{\bullet}="m",
(10,-8) *{\bullet}="n",
\ar @{.} "x" ; "xx"
\ar @{-}^{t} "a" ; "b"
\ar @{-}_<<<{a_{p}} "i" ; "a"
\ar @{-}_<<<{a_{p-1}} "j" ; "a"
\ar @{-} "k" ; "a"
\ar @{-} "l" ; "a"
\ar @{-}^<<<{a_{2}} "m" ; "a"
\ar @{-}^<<<{a_{1}} "n" ; "a"
\ar @{-}_<<<<<<<<{a_{p,q}} "i" ; "ii"
\ar @{-}^{a_{p,2}} "i" ; "iii"
\ar @/_/ @{.} "kk" ; "ll"
\ar @/_/ @{.} "i-i" ; "i-ii"
\POS(-3,0)*+{x}
\POS(23,0)*+{y}
\POS(12,10)*+{z}
\end{xy}
\]
with $p, q \ge 1$, where $a_{p}=a_{p,1}$. Turn the edge $t$ anti-clockwise around the vertex $x$ and select the edge $a_{p}$ which $t$ first meets. Then select the vertex $z$ of the edge $a_{p}$ different from $x$. Add a new edge $t^{\prime}$ connecting the vertices $z$ and $y$, and remove the edge $t$. As a consequence, we get the following Brauer tree $B^{\prime}$:
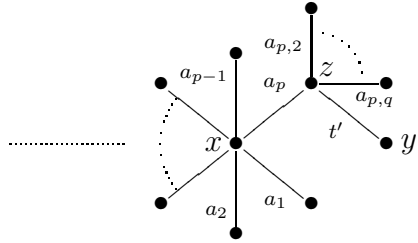
\[
\begin{xy}
(-15,0) *{}="x",
(-30,0) *{}="xx",
(0,0) *{\bullet}="a",
(20,0) *{\bullet}="b",
(10,8) *{\bullet}="i",
(20,8) *{\bullet}="ii",
(10,18) *{\bullet}="iii",
(17,8) *{}="i-i",
(10,15) *{}="i-ii",
(0,12) *{\bullet}="j",
(-10,8) *{\bullet}="k",
(-8,6) *{}="kk",
(-10,-8) *{\bullet}="l",
(-8,-6) *{}="ll",
(0,-12) *{\bullet}="m",
(10,-8) *{\bullet}="n",
\ar @{.} "x" ; "xx"
\ar @{-}_{t^{\prime}} "i" ; "b"
\ar @{-}_<<<{a_{p}} "i" ; "a"
\ar @{-}_<<<{a_{p-1}} "j" ; "a"
\ar @{-} "k" ; "a"
\ar @{-} "l" ; "a"
\ar @{-}^<<<{a_{2}} "m" ; "a"
\ar @{-}^<<<{a_{1}} "n" ; "a"
\ar @{-}_<<<<<<<<{a_{p,q}} "i" ; "ii"
\ar @{-}^{a_{p,2}} "i" ; "iii"
\ar @/_/ @{.} "kk" ; "ll"
\ar @/_/ @{.} "i-i" ; "i-ii"
\POS(-3,0)*+{x}
\POS(23,0)*+{y}
\POS(12,10)*+{z}
\end{xy}
\]
Next, assume that $t$ is not an end edge of $B$:  
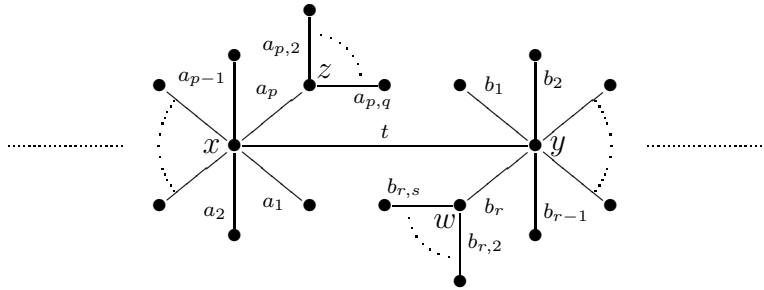
\[
\begin{xy}
(-15,0) *{}="x",
(-30,0) *{}="xx",
(55,0) *{}="y",
(70,0) *{}="yy",
(0,0) *{\bullet}="a",
(40,0) *{\bullet}="b",
(30,8) *{\bullet}="c",
(40,12) *{\bullet}="d",
(50,8) *{\bullet}="e",
(48,6) *{}="ee",
(50,-8) *{\bullet}="f",
(48,-6) *{}="ff",
(40,-12) *{\bullet}="g",
(30,-8) *{\bullet}="h",
(20,-8) *{\bullet}="hh",
(30,-18) *{\bullet}="hhh",
(23,-8) *{}="h-h",
(30,-15) *{}="h-hh",
(10,8) *{\bullet}="i",
(20,8) *{\bullet}="ii",
(10,18) *{\bullet}="iii",
(17,8) *{}="i-i",
(10,15) *{}="i-ii",
(0,12) *{\bullet}="j",
(-10,8) *{\bullet}="k",
(-8,6) *{}="kk",
(-10,-8) *{\bullet}="l",
(-8,-6) *{}="ll",
(0,-12) *{\bullet}="m",
(10,-8) *{\bullet}="n",
\ar @{.} "x" ; "xx"
\ar @{.} "y" ; "yy"
\ar @{-}^{t} "a" ; "b"
\ar @{-}^<<<{b_{1}} "c" ; "b"
\ar @{-}^<<<{b_{2}} "d" ; "b"
\ar @{-} "e" ; "b"
\ar @{-} "f" ; "b"
\ar @{-}_<<<{b_{r-1}} "g" ; "b"
\ar @{-}_<<<{b_{r}} "h" ; "b"
\ar @/^/ @{.} "ee" ; "ff"
\ar @{-}_<<<<{a_{p}} "i" ; "a"
\ar @{-}_<<<{a_{p-1}} "j" ; "a"
\ar @{-} "k" ; "a"
\ar @{-} "l" ; "a"
\ar @{-}^<<<{a_{2}} "m" ; "a"
\ar @{-}^<<<{a_{1}} "n" ; "a"
\ar @{-}_<<<<<<<<{a_{p,q}} "i" ; "ii"
\ar @{-}^{a_{p,2}} "i" ; "iii"
\ar @{-}_<<<<<<<{b_{r,s}} "h" ; "hh"
\ar @{-}^{b_{r,2}} "h" ; "hhh"
\ar @/_/ @{.} "kk" ; "ll"
\ar @/_/ @{.} "i-i" ; "i-ii"
\ar @/_/ @{.} "h-h" ; "h-hh"
\POS(-3,0)*+{x}
\POS(43,0)*+{y}
\POS(12,10)*+{z}
\POS(28,-10)*+{w}
\end{xy}
\]
with $p, q, r, s \ge 1$, where $a_{p}=a_{p,1}, b_{r}=b_{r,1}$. Turn the edge $t$ anti-clockwise around the vertex $x$ and select the edge $a_{p}$ which $t$ first meets. Then select the vertex $z$ of the edge $a_{p}$ different from $x$. Similarly, turn the edge $t$ anti-clockwise around the vertex $y$ and select the edge $b_{r}$ which $t$ first meets. Then select the vertex $w$ of the edge $b_{r}$ different from $y$. Add a new edge $t^{\prime}$ connecting the vertices $z$ and $w$, and remove the edge $t$. As a consequence, we get the following Brauer tree $B^{\prime}$: 
\[
\begin{xy}
(-15,0) *{}="x",
(-30,0) *{}="xx",
(55,0) *{}="y",
(70,0) *{}="yy",
(0,0) *{\bullet}="a",
(40,0) *{\bullet}="b",
(30,8) *{\bullet}="c",
(40,12) *{\bullet}="d",
(50,8) *{\bullet}="e",
(48,6) *{}="ee",
(50,-8) *{\bullet}="f",
(48,-6) *{}="ff",
(40,-12) *{\bullet}="g",
(30,-8) *{\bullet}="h",
(20,-8) *{\bullet}="hh",
(30,-18) *{\bullet}="hhh",
(23,-8) *{}="h-h",
(30,-15) *{}="h-hh",
(10,8) *{\bullet}="i",
(20,8) *{\bullet}="ii",
(10,18) *{\bullet}="iii",
(17,8) *{}="i-i",
(10,15) *{}="i-ii",
(0,12) *{\bullet}="j",
(-10,8) *{\bullet}="k",
(-8,6) *{}="kk",
(-10,-8) *{\bullet}="l",
(-8,-6) *{}="ll",
(0,-12) *{\bullet}="m",
(10,-8) *{\bullet}="n",
\ar @{.} "x" ; "xx"
\ar @{.} "y" ; "yy"
\ar @{-}^{t^{\prime}} "i" ; "h"
\ar @{-}^<<<{b_{1}} "c" ; "b"
\ar @{-}^<<<{b_{2}} "d" ; "b"
\ar @{-} "e" ; "b"
\ar @{-} "f" ; "b"
\ar @{-}_<<<{b_{r-1}} "g" ; "b"
\ar @{-}_<<<{b_{r}} "h" ; "b"
\ar @/^/ @{.} "ee" ; "ff"
\ar @{-}_<<<<{a_{p}} "i" ; "a"
\ar @{-}_<<<{a_{p-1}} "j" ; "a"
\ar @{-} "k" ; "a"
\ar @{-} "l" ; "a"
\ar @{-}^<<<{a_{2}} "m" ; "a"
\ar @{-}^<<<{a_{1}} "n" ; "a"
\ar @{-}_<<<<<<<<{a_{p,q}} "i" ; "ii"
\ar @{-}^{a_{p,2}} "i" ; "iii"
\ar @{-}_<<<<<<<{b_{r,s}} "h" ; "hh"
\ar @{-}^{b_{r,2}} "h" ; "hhh"
\ar @/_/ @{.} "kk" ; "ll"
\ar @/_/ @{.} "i-i" ; "i-ii"
\ar @/_/ @{.} "h-h" ; "h-hh"
\POS(-3,0)*+{x}
\POS(43,0)*+{y}
\POS(12,10)*+{z}
\POS(28,-10)*+{w}
\end{xy}
\]

\begin{corollary}\label{edge_reflection}
The Brauer tree $B^{\prime}$ determines $\Gamma$. 
\end{corollary}

\begin{corollary}[cf. {\cite[Theorem 3.7]{AH1}}]\label{derived_eq_BT}
There exists a sequence of Brauer tree algebras $\Lambda=\Delta_{0}, \Delta_{1}, \cdots, \Delta_{l}$ such that $\Delta_{i+1}$ is the reflection of $\Delta_{i}$ at a suitable vertex for $0 \le i < l$ and $\Delta_{l}$ is a Brauer line algebra, i.e., the path algebra defined by the quiver 
\[
\xymatrix{
 1 \ar [r]<2pt>^{\alpha_{1}} & 2 \ar [r]<2pt>^{\alpha_{2}} \ar [l]<2pt>^{\beta_{1}} & \cdots \ar [l]<2pt>^{\beta_{2}}  \ar [r]<2pt>^{\alpha_{n-2}} & n-1 \ar [r]<2pt>^>>>>>{\alpha_{n-1}} \ar [l]<2pt>^{\beta_{n-2}}  & n \ar [l]<2pt>^>>>>>{\beta_{n-1}} 
}
\]
with relations
\[
\alpha_{i+1}\alpha_{i} = \beta_{i}\beta_{i+1}=0, \quad \alpha_{i}\beta_{i}=\beta_{i+1}\beta_{i}
\]
for $1 \le i < n-1$, where $n$ is the number of vertices of $(\Lambda_{0}, \Lambda_{1}, \rho)$. 
\end{corollary}
\begin{proof}
According to Theorem \ref{reflection_quiver}, the reflection for a Brauer tree algebra $\Lambda$ at a vertex $t$ reduces the length of cycles of $(\Lambda_{0}, \Lambda_{1}, \rho)$ including $t$.
\end{proof}

\begin{example}\label{reflection_ex}
We have the following transformations for Brauer trees:
\[
\begin{array}{ccccc}
\begin{xy}
(0,0) *{\bullet}="a",
(-10,-7) *{\bullet}="b",
(-6,-18) *{\bullet}="c",
(6,-18) *{\bullet}="d",
(0,-10) *{\bullet}="f",
(10,-7) *{\bullet}="e",

\ar @{-}^{1} "a" ; "f"
\ar @{-}^{2} "b" ; "f"
\ar @{-}^{3} "c" ; "f"
\ar @{-}^{4} "d" ; "f"
\ar @{-}^{5} "e" ; "f"
\end{xy}
&
\begin{xy}
(-7,-7) *{}="a",
(7,-7) *{}="b",
\SelectTips{eu}{12} \ar @{->}^{\text{reflection}}_{\text{at}\ 1} "a" ; "b"
\end{xy}
& 
\begin{xy}
(0,0) *{\bullet}="a",
(-10,-7) *{\bullet}="b",
(-6,-18) *{\bullet}="c",
(6,-18) *{\bullet}="d",
(0,-10) *{\bullet}="f",
(10,-7) *{\bullet}="e",

\ar @{-}_{1^{\prime}} "a" ; "b"
\ar @{-}^{2} "b" ; "f"
\ar @{-}^{3} "c" ; "f"
\ar @{-}^{4} "d" ; "f"
\ar @{-}^{5} "e" ; "f"
\end{xy}
&
\begin{xy}
(-1,-7) *{}="a",
(1,-7) *{}="b",
\ar @{=} "a" ; "b"
\end{xy}
& 
\begin{xy}
(-10,-7) *{\bullet}="a",
(0,-7) *{\bullet}="b",
(10,-7) *{\bullet}="c",
(20,-7) *{\bullet}="d",
(10,3) *{\bullet}="e",
(10,-17) *{\bullet}="f",

\ar @{-}^{1^{\prime}} "a" ; "b"
\ar @{-}^{2} "b" ; "c"
\ar @{-}_{5} "c" ; "e"
\ar @{-}_{4} "c" ; "d"
\ar @{-}_{3} "c" ; "f"
\end{xy}
\end{array}
\]
\[
\begin{array}{ccccc}
\begin{xy}
(-8,-7) *{\bullet}="a",
(0,-7) *{\bullet}="b",
(8,-7) *{\bullet}="c",
(16,-7) *{\bullet}="d",
(8,1) *{\bullet}="e",
(8,-15) *{\bullet}="f",
\ar @{-}^{1} "a" ; "b"
\ar @{-}^{2} "b" ; "c"
\ar @{-}_{3} "c" ; "e"
\ar @{-}_{4} "c" ; "d"
\ar @{-}_{5} "c" ; "f"
\end{xy}
&
\begin{xy}
(-7,-7) *{}="a",
(7,-7) *{}="b",
\SelectTips{eu}{12} \ar @{->}^{\text{reflection}}_{\text{at}\ 3} "a" ; "b"
\end{xy}
& 
\begin{xy}
(-8,-7) *{\bullet}="a",
(0,-7) *{\bullet}="b",
(8,-7) *{\bullet}="c",
(16,-7) *{\bullet}="d",
(8,1) *{\bullet}="e",
(8,-15) *{\bullet}="f",
\ar @{-}^{1} "a" ; "b"
\ar @{-}_{2} "b" ; "c"
\ar @{-}^{3^{\prime}} "b" ; "e"
\ar @{-}^{4} "c" ; "d"
\ar @{-}^{5} "c" ; "f"
\end{xy}
&
\begin{xy}
(-1,-7) *{}="a",
(1,-7) *{}="b",
\ar @{=} "a" ; "b"
\end{xy}
& 
\begin{xy}
(-8,0) *{\bullet}="a",
(-8,-14) *{\bullet}="b",
(-1,-7) *{\bullet}="c",
(9,-7) *{\bullet}="d",
(16,0) *{\bullet}="e",
(16,-14) *{\bullet}="f",
\ar @{-}^{3^{\prime}} "a" ; "c"
\ar @{-}_{1} "b" ; "c"
\ar @{-}^{2} "c" ; "d"
\ar @{-}^{4} "d" ; "e"
\ar @{-}_{5} "d" ; "f"
\end{xy}
\end{array}
\]
\[
\begin{array}{ccccc}
\begin{xy}
(-8,0) *{\bullet}="a",
(-8,-14) *{\bullet}="b",
(-1,-7) *{\bullet}="c",
(9,-7) *{\bullet}="d",
(16,0) *{\bullet}="e",
(16,-14) *{\bullet}="f",
\ar @{-}^{1} "a" ; "c"
\ar @{-}_{2} "b" ; "c"
\ar @{-}^{3} "c" ; "d"
\ar @{-}^{4} "d" ; "e"
\ar @{-}_{5} "d" ; "f"
\end{xy}
&
\begin{xy}
(-7,-7) *{}="a",
(7,-7) *{}="b",
\SelectTips{eu}{12} \ar @{->}^{\text{reflection}}_{\text{at}\ 1} "a" ; "b"
\end{xy}
& 
\begin{xy}
(-8,0) *{\bullet}="a",
(-8,-14) *{\bullet}="b",
(-1,-7) *{\bullet}="c",
(9,-7) *{\bullet}="d",
(16,0) *{\bullet}="e",
(16,-14) *{\bullet}="f",
\ar @{-}_{1^{\prime}} "a" ; "b"
\ar @{-}_{2} "b" ; "c"
\ar @{-}^{3} "c" ; "d"
\ar @{-}^{4} "d" ; "e"
\ar @{-}_{5} "d" ; "f"
\end{xy}
&
\begin{xy}
(-1,-7) *{}="a",
(1,-7) *{}="b",
\ar @{=} "a" ; "b"
\end{xy}
& 
\begin{xy}
(-7,-7) *{\bullet}="a",
(0,-7) *{\bullet}="b",
(7,-7) *{\bullet}="c",
(14,-7) *{\bullet}="d",
(21,-7) *{\bullet}="e",
(14,1) *{\bullet}="f",
\ar @{-}_{1^{\prime}} "a" ; "b"
\ar @{-}_{2} "b" ; "c"
\ar @{-}_{3} "c" ; "d"
\ar @{-}_{5} "d" ; "e"
\ar @{-}_{4} "d" ; "f"
\end{xy}
\end{array}
\]
\[
\begin{array}{ccccc}
\begin{xy}
(-6.5,-7) *{\bullet}="a",
(0,-7) *{\bullet}="b",
(6.5,-7) *{\bullet}="c",
(13,-7) *{\bullet}="d",
(19.6,-7) *{\bullet}="e",
(13,1) *{\bullet}="f",
\ar @{-}_{1} "a" ; "b"
\ar @{-}_{2} "b" ; "c"
\ar @{-}_{3} "c" ; "d"
\ar @{-}_{5} "d" ; "e"
\ar @{-}_{4} "d" ; "f"
\end{xy}
&
\begin{xy}
(-7,-7) *{}="a",
(7,-7) *{}="b",
\SelectTips{eu}{12} \ar @{->}^{\text{reflection}}_{\text{at}\ 4} "a" ; "b"
\end{xy}
& 
\begin{xy}
(-7,-7) *{\bullet}="a",
(0,-7) *{\bullet}="b",
(7,-7) *{\bullet}="c",
(14,-7) *{\bullet}="d",
(21,-7) *{\bullet}="e",
(14,1) *{\bullet}="f",
\ar @{-}_{1} "a" ; "b"
\ar @{-}_{2} "b" ; "c"
\ar @{-}_{3} "c" ; "d"
\ar @{-}^{4^{\prime}} "c" ; "f"
\ar @{-}_{5} "d" ; "e"
\end{xy}
&
\begin{xy}
(-1,-7) *{}="a",
(1,-7) *{}="b",
\ar @{=} "a" ; "b"
\end{xy}
& 
\begin{xy}
(-7,-7) *{\bullet}="a",
(0,-7) *{\bullet}="b",
(7,-7) *{\bullet}="c",
(14,-7) *{\bullet}="d",
(21,-7) *{\bullet}="e",
(7,1) *{\bullet}="f",
\ar @{-}_{1} "a" ; "b"
\ar @{-}_{2} "b" ; "c"
\ar @{-}_{3} "c" ; "d"
\ar @{-}_{5} "d" ; "e"
\ar @{-}^{4^{\prime}} "c" ; "f"
\end{xy}
\end{array}
\]
\[
\begin{array}{ccccc}
\begin{xy}
(-6.5,-7) *{\bullet}="a",
(0,-7) *{\bullet}="b",
(6.5,-7) *{\bullet}="c",
(13,-7) *{\bullet}="d",
(19.5,-7) *{\bullet}="e",
(6.5,1) *{\bullet}="f",
\ar @{-}_{1} "a" ; "b"
\ar @{-}_{2} "b" ; "c"
\ar @{-}_{3} "c" ; "d"
\ar @{-}_{4} "d" ; "e"
\ar @{-}^{5} "c" ; "f"
\end{xy}
&
\begin{xy}
(-7,-7) *{}="a",
(7,-7) *{}="b",
\SelectTips{eu}{12} \ar @{->}^{\text{reflection}}_{\text{at}\ 3} "a" ; "b"
\end{xy}
& 
\begin{xy}
(-6.5,-7) *{\bullet}="a",
(0,-7) *{\bullet}="b",
(6.5,-7) *{\bullet}="c",
(13,-7) *{\bullet}="d",
(19.5,-7) *{\bullet}="e",
(6.5,1) *{\bullet}="f",
\ar @{-}_{1} "a" ; "b"
\ar @{-}_{2} "b" ; "c"
\ar @{-}^{3^{\prime}} "f" ; "d"
\ar @{-}_{4} "d" ; "e"
\ar @{-}^{5} "c" ; "f"
\end{xy}
&
\begin{xy}
(-1,-7) *{}="a",
(1,-7) *{}="b",
\ar @{=} "a" ; "b"
\end{xy}
& 
\begin{xy}
(-6,-7) *{\bullet}="a",
(0,-7) *{\bullet}="b",
(6,-7) *{\bullet}="c",
(12,-7) *{\bullet}="d",
(18,-7) *{\bullet}="e",
(24,-7) *{\bullet}="f",
\ar @{-}_{1} "a" ; "b"
\ar @{-}_{2} "b" ; "c"
\ar @{-}_{5} "c" ; "d"
\ar @{-}_{3^{\prime}} "d" ; "e"
\ar @{-}_{4} "e" ; "f"
\end{xy}
\end{array}
\]
\end{example}


\vspace{5pt}
\begin{flushleft}
Tokyo Keizai University \\
1-7-34, Minami-cho, Kokubunji-shi, Tokyo 185-8502, Japan \\
{\it E-mail address} : {\rm abeh@tku.ac.jp} 
\end{flushleft}

\end{document}